\documentclass[a4paper]{article}

\usepackage{arxiv}

\usepackage{bookmark}
\usepackage{booktabs}
\usepackage{hyperref}
\usepackage{graphicx}
\usepackage{subfigure}
\usepackage{amsmath,amsthm}
\usepackage{amsfonts}
\usepackage{enumerate}
\usepackage[nosort]{cite}
\usepackage{url}

\usepackage{xcolor}

\newcommand{\Fref}[1]{Fig.~\ref{#1}}
\newcommand{\Tref}[1]{Table~\ref{#1}}

\newtheorem{proposition}{Proposition}
\newtheorem{remark}{Remark}

\renewcommand{\shorttitle}{Flow Measurement: An Inverse Problem Formulation}

\title{Flow Measurement: An Inverse Problem Formulation }

\author{Jiwei Li\\
	Yau Mathematical Sciences Center\\
	Tsinghua University\\
	Beijing 100084, China\\
	\href{mailto:li-jw20@mails.tsinghua.edu.cn}{li-jw20@mails.tsinghua.edu.cn}\\
	\And Lingyun Qiu\\
	Yau Mathematical Sciences Center\\
	Tsinghua University\\
	Beijing 100084, China\\
	\href{mailto:lyqiu@tsinghua.edu.cn}{lyqiu@tsinghua.edu.cn}\\
	\And Zhongjing Wang\\
	Department of Hydraulic Engineering\\
	Tsinghua University\\
	Beijing 100084, China\\
	\href{mailto:zj.wang@tsinghua.edu.cn}{zj.wang@tsinghua.edu.cn}
	\And Hui Yu\\
	Yau Mathematical Sciences Center\\
	Tsinghua University\\
	Beijing 100084, China\\
	\href{mailto:huiyu@tsinghua.edu.cn}{huiyu@tsinghua.edu.cn}\\}

\begin{document}
	
	\maketitle
	
	\begin{abstract}
		This paper proposes a new mathematical formulation for flow measurement based on the inverse source problem for wave equations with partial boundary measurement. Inspired by the design of acoustic Doppler current profilers (ADCPs), we formulate an inverse source problem that can recover the flow field from the observation data on a few boundary receivers. To our knowledge, this is the first mathematical model of flow measurement using partial differential equations. This model is proved well-posed, and the corresponding algorithm is derived to compute the velocity field efficiently. Extensive numerical simulations are performed to demonstrate the accuracy and robustness of our model. Our formulation is capable of simulating a variety of practical measurement scenarios.
	\end{abstract}
	
	
	
	\section{Introduction\label{sec:intro}}
	Flow measurement is the process of measuring the dynamics of flow in vessels, pipelines or open channels. Specifically, flow measurement involves determining the mass flow rate, the volumetric flow rate, the flow velocity, and the Reynolds tensor, among which flow velocity is the significant and primary concern. It is a key technology for experimental fluid dynamics and has been studied extensively. Flow measurement has a wide range of applications in either industrial processes or daily life, including the nuclear reactor fluid flow measurement, the control of fuel flow in engine management systems, the regulation of drug delivery in ventilators, the measurement of corrosive materials in chemical industries and balance management in the water supply network\cite{tongAnalysisFlowDistribution2021,termaatFluidFlowMeasurements1970,tudosieAircraftGasTurbineEngine2011,shademaniActiveRegulationOnDemand2017,beraFlowMeasurementTechnique2012,klingelReviewWaterBalance2015}.
	
	Despite its great practical importance, no fully satisfactory flow measurement theory has yet been developed. Over the past decades, many techniques have been developed to measure the flow velocity, such as Pitot tubes, hot wire film anemometer (HWFA), ultrasonic flowmeters (such as transit time flowmeters and Doppler flowmeters), and magnetic flowmeters\cite{liptakInstrumentEngineersHandbook2003c,blincoTurbulenceCharacteristicsFree1971b,nakagawaPredictionContributionsReynolds1977}. However, these methods still have significant limitations. We mention some typical ones here, according to the authors' best knowledge.
	\begin{itemize}
		\item The major limitations of Pitot tubes are poor accuracy, low rangeability, sensitivity to flow direction disturbances, the requirement for high-velocity flow, and its inability to be used with viscous or dirty fluids.
		\item The accuracy of the HWFA is easily affected by the ambient temperature, which is the main source of errors. Besides, HWFA is a contact method and interferes the original flow field greatly.
		\item Transit time flowmeters suffer from pipe-wall interference, causing accuracy problems. They also cannot be used on dirty fluids with bubbles. Moreover, these techniques cannot perform the instantaneous multipoint measurement.
		\item Although the Doppler flowmeter has the advantage of its non-invasive design, it cannot be applied with clear fluid. Another main limitation is the assumption that the flow field over the measurement volume is homogeneous.  
	\end{itemize}
	
	This paper aims to provide a mathematical model based on partial differential equations (PDEs) for measuring flow using ultrasound. We first recast the flow measurement based on the Doppler frequency shift effect as an inverse source problem for the wave equation. Then, we obtain the instantaneous velocity of particles, i.e. velocity field, in a large region by solving the associated optimization problem. Considering measuring methods based on PDEs looks pretty appealing, especially in contexts when the flow is complex. Indeed, all previous methodologies leverage a homogeneity idea, while the PDE approach involves individual tracking of natural mark points in the flow. Thus, it can provide high-resolution velocity measurement for complex flows, such as vortices, tidal flows, and complex intersections at stream junctions as shown in \Fref{junction}.
	\begin{figure}
		\centering
		\includegraphics[width=0.5\textwidth]{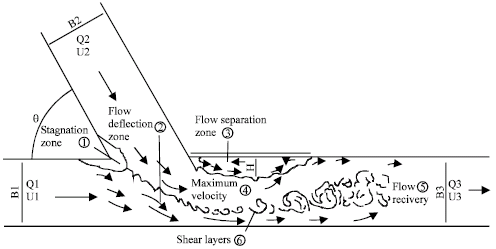}
		\caption{Observed zone at the junction of two streams\cite{bestFlowDynamicsRiver1987}.}
		\label{junction}
	\end{figure}
	
	The inverse source problem is an important subject, and has various significant applications which have been extensively studied, such as radar imaging, pollution detection, magnetoencephalography, photoacoustic tomography/thermoacoustic tomography (PAT/TAT), and through-the-wall imaging\cite{cheneyFundamentalsRadarImaging2009,andrleInverseSourceProblem2015,linTheoreticalNumericalStudies2022,ammariInverseSourceProblem2002,liThroughWallDetectionHuman2012,kuchmentMathematicsPhotoacousticThermoacoustic2015,burgholzerExactApproximativeImaging2007,zangerlFullFieldInversion2019,ellerMicrolocallyAccurateSolution2020,hristovaReconstructionTimeReversal2008}. This class of inverse problems aims to reconstruct the source term (or the inhomogeneous term), including locations and magnitudes of sources, from data usually measured on the boundary. It is expected that such measurements can be used to detect anomalies in the medium since the moving impurities in the fluid can be considered as environmental abnormalities, causing the scattering of the ultrasonic waves. Therefore, the inverse source problem provides a suitable formulation for the flow measurement method with ultrasonic waves.
	
	
	
	
	Flexible selections of parameter settings in our novel model allow for testing and verification of various configurations of instruments. We propose to test several typical scenarios closely related to practical measurements in this paper. These examples will demonstrate the performances of our formulation when varying the dominant frequency of the source wave, modifying the layouts of receivers, and adding noise to registered data. Indeed, our model can be applied to more scenarios. 
	
	
	This paper is organized as follows. In section \ref{sec:model}, several common measuring instruments and their measuring principles are reviewed. We propose a mathematical formulation for flow measurement in the form of an inverse source problem in section \ref{sec:directInv}. Section \ref{sec:numer results} consists of the numerical method to solve the inverse problem and several numerical examples in various experimental settings, demonstrating the high accuracy of our new model. Finally, in section \ref{sec:conclusion} some conclusions and an outlook on future work are given.
	
	\section{Mathematical model for the flow measurement\label{sec:model}}
	\subsection{Physical process of the flow measurement}
	At the beginning of this section, we review some typical instruments for flow measurement. The well-known Pitot tube, invented by a French engineer, Henri Pitot in 1732, was widely used to estimate the flow speed based on Bernoulli's equation. However, the result is of low accuracy, and this method can only obtain the velocity measurement at the Pitot tube entrance.
	
	
	
	Another primary class of instruments in the flow measurement is based on the changes in wave properties when ultrasonic waves propagate in the flow, such as the wave velocity, the wave phase, and the wave frequency; see \Fref{diff} which shows how the wave propagation speed changes from the transmitters $T$ to the receivers $R$. 
	
	For example, the flow velocity is approximated by $c^2/(2L) \cdot \Delta t$	when the change in wave velocity captured by the receiver is taken into account.
	Here, $\Delta t$ is the time difference for the wave to propagate upstream and downstream, $c\approx 1500~\mathrm{m/s}$ is the velocity of ultrasonic sound in still water, and $L$ is the distance between the wave transmitter and receiver. 
	It should be noted that the large constant $c^2/(2L)$ can magnify the measurement error of $\Delta t$ significantly. 
	Therefore this method requires high accuracy of measuring circuit in order to obtain a good approximation on $u$. Usually, the error of $\Delta t$ must be less than $0.01~\mathrm{\mu s}$ for the result to reach the accuracy of $1\%$. In addition, since $c$ varies in different flows, choosing $c$ as a constant in the calculation formula above also causes errors. Moreover, only averaged velocity can be obtained by these methods.
	
	\begin{figure}
		\centering
		\begin{minipage}{0.48\textwidth}
			\flushright
			\subfigure[Measurement instruments based on the changes in wave properties]{\includegraphics[width=0.6\columnwidth]{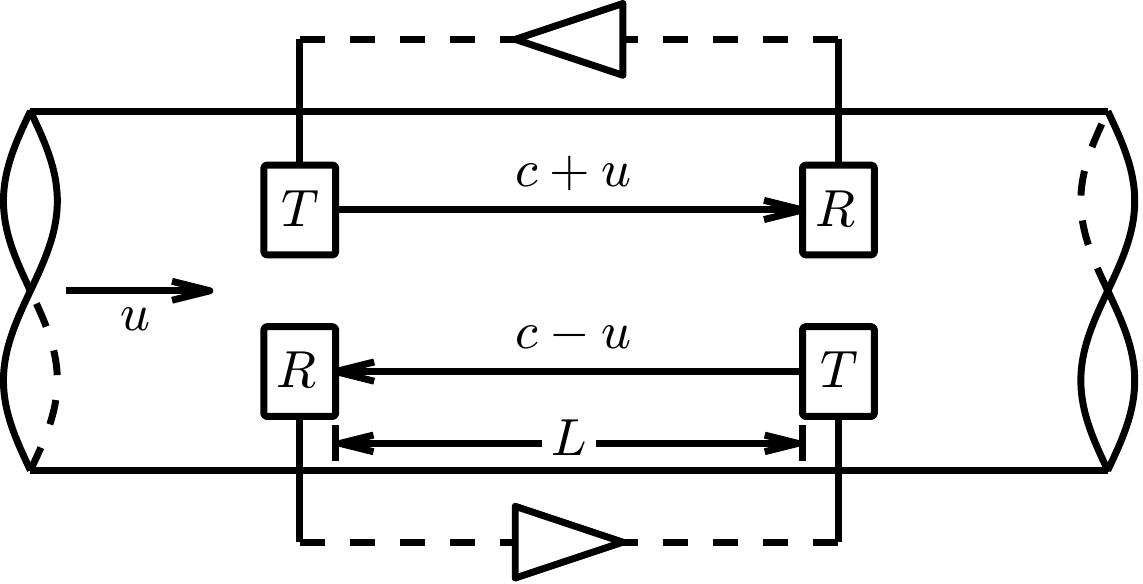}
				\label{diff}}
		\end{minipage}
		\begin{minipage}{0.48\textwidth}
			\flushleft
			\subfigure[Acoustic Doppler current profiler \cite{sellarHighresolutionVelocimetryEnergetic2015}]{\includegraphics[width=0.6\columnwidth]{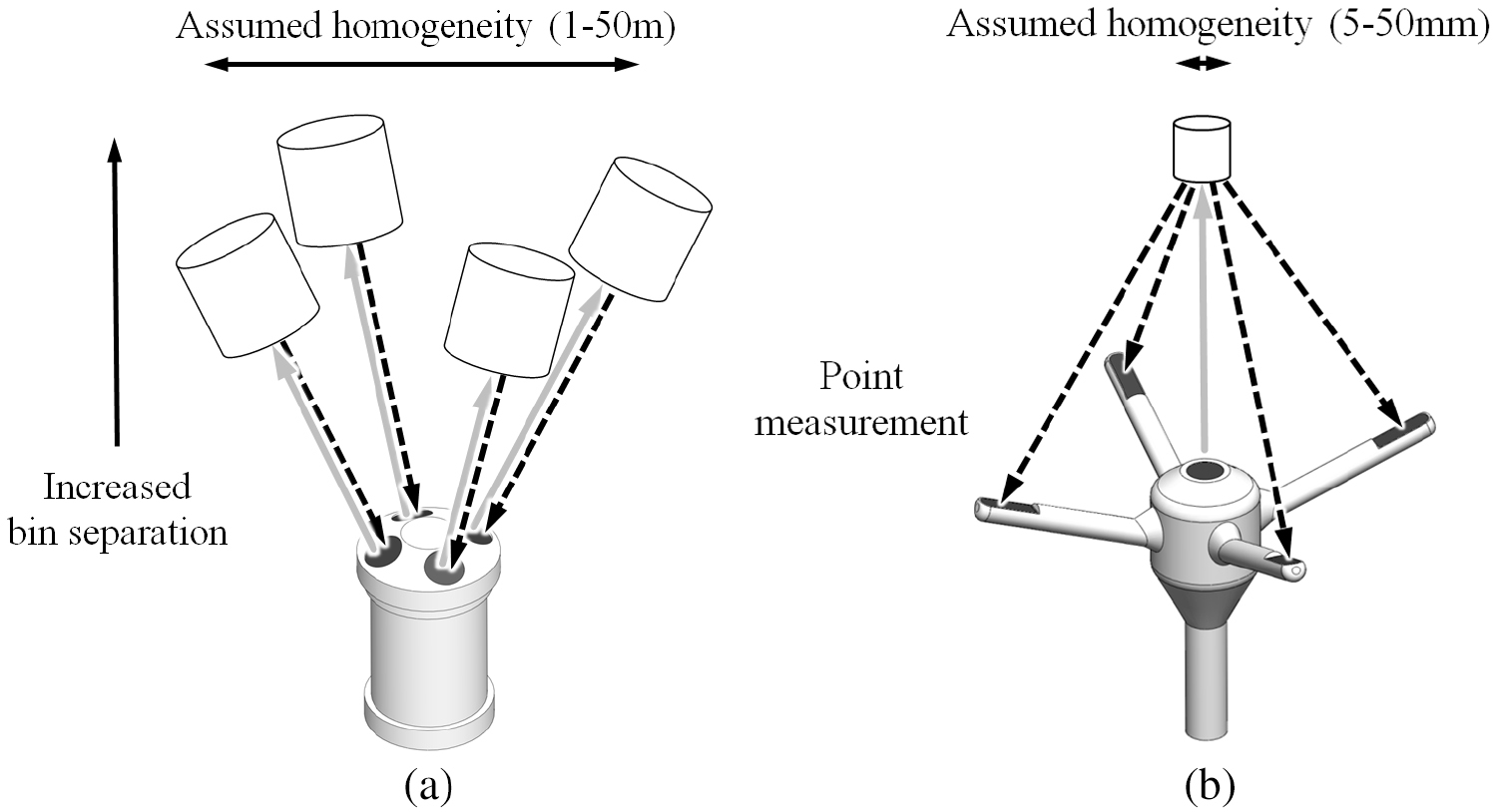}
				\label{adcp}}
		\end{minipage}
		
		\caption{Measurement instruments}
	\end{figure}
	
	
	
	The acoustic Doppler velocimeter \big(ADV in the right figure of \Fref{adcp}; see \cite{lohrmannAcousticDopplerVelocimeterADV1994,lohrmannDirectMeasurementsReynolds1995}\big) is a device to record instantaneous velocity components at a single point with a relatively high-frequency ultrasonic sound by applying the Doppler frequency shift effect. Velocity measurements are performed by measuring particle velocity in a relatively small sampling volume located 5 or 10 cm from ADV. However, it is unsuitable for measuring the overall velocity distribution in a large environment.
	
	The acoustic Doppler current profiler \big(ADCP in the left figure of \Fref{adcp}; see \cite{stoneEvaluatingVelocityMeasurement2007,tokyayInvestigationTwoElemental2009,whippleMeasurementsReynoldsStress2006,staceyObservationsTurbulencePartially1999}\big) is widely used in various environments, especially in oceans, rivers, canals and streams. 
	ADCP uses shifted frequency of echoes scattered by the particles in sampling volumes to calculate the projected velocities of the flow along the direction of the beams. There are several main limitations of this technique. 
	Firstly, ADCP does not work well without the assumption that the flow field over the measured volume is homogeneous. So its validity is likely to fail in practical flow conditions. Secondly, the high-resolution measurement far from ADCP cannot realize due to the averaging over the large sampling volumes.
	
	
	The convergent-beam acoustic Doppler current profiler (C-ADCP, \cite{sellarHighresolutionVelocimetryEnergetic2015}) developed by Sellar et al. \cite{sellarHighresolutionVelocimetryEnergetic2015} is able to overcome the constraint on the short distance between the sampling volume and the advice, and the drawback due to the diverging beams of ADCP. 
	More available receivers and more complex configurations enable C-ADCP's focal measurements to significantly reduce Doppler noise and effectively improve the quality of measured data.
	
	
	
	We first investigate the physical process of the device ADCP, which enjoys a relatively high flow measurement accuracy in a sampling volume; see \Fref{phyProc}. ADCP contains several transducers to transmit and receive ultrasonic sound waves. Without loss of generality, we analyze the physical process for the case of a single transducer rather than the typical 4-beam configuration. 
	Then we explain how ADCP works based on the configuration shown in \Fref{phyProc}. 
	ADCP assumes that the density of particles is close to the density of the water flow; accordingly, the velocity of the particles in the flow equals the flow velocity. Hence they perform well in the flow with particles such as sand, gas bubbles, mist, or other impurities that can scatter the ultrasonic wave. 
	Each transducer acts as both transmitter and receiver of ultrasonic waves. Transducers emit a pulsed acoustic signal with a specific frequency along a directed beam. Then the acoustic waves are scattered by the moving particles and propagate back to the surrounding flow. Finally, these echoes are received by the corresponding transducer. 
	Thanks to the Doppler shift effect, the frequency shift of the echoes is proportional to the flow velocity along the acoustic beam.
	We calculate the projected velocity along the beam using the frequency difference of the emitted waves and echoes according to the Doppler frequency shift formula.
	More precisely, since the sound speed $c$ is much larger than the projected flow speed $\hat{u}$ along the beam, the Doppler frequency shift along each beam is approximately given by
	\begin{equation}
		\Delta q\approx \frac{2\hat{u}\sin\phi}{c}q_s,
		\label{dopplerinadcp}
	\end{equation}
	where $\phi$ is the angle between the beam and the vertical direction. However, the algebraic formula \eqref{dopplerinadcp} is oversimplified, resulting in a low-accuracy flow velocity measurement with ADCP.
	Furthermore, we can only obtain the averaged velocity of a sampling volume in the flow, and it cannot capture the velocity field pointwise or far from the instruments. In other words, it cannot simultaneously obtain velocities at multiple locations in the flow during a single measurement process.
	
	\subsection{Reformulation of the flow measurement}\label{SecReform}
	
	We present a novel mathematical model that enables us to overcome the limitations of low spatial resolution, small measurement region, and lack of instantaneous multi-point measurement. 
	Assume that the transmitter and receiver do not have to be deployed in the same location. One can decompose the measurement into two processes; see \Fref{phyProc}.
	
	\begin{figure}
		\centering
		\begin{minipage}{0.48\textwidth}
			\flushright
			\subfigure[Process 1]{
				\includegraphics[width=0.6\textwidth]{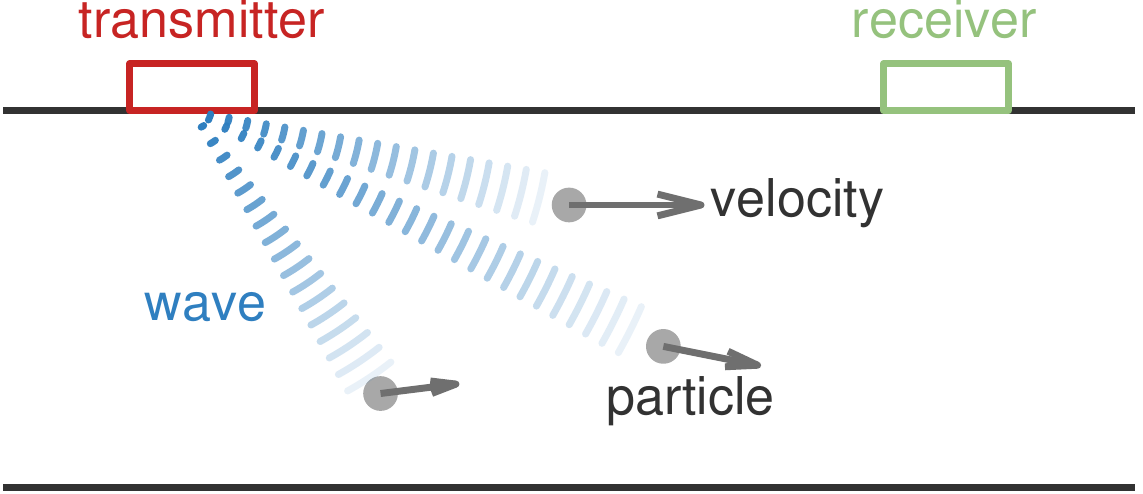}
				\label{subProc1}}
		\end{minipage}
		\begin{minipage}{0.48\textwidth}
			\flushleft
			\subfigure[Process 2]{
				\includegraphics[width=0.6\textwidth]{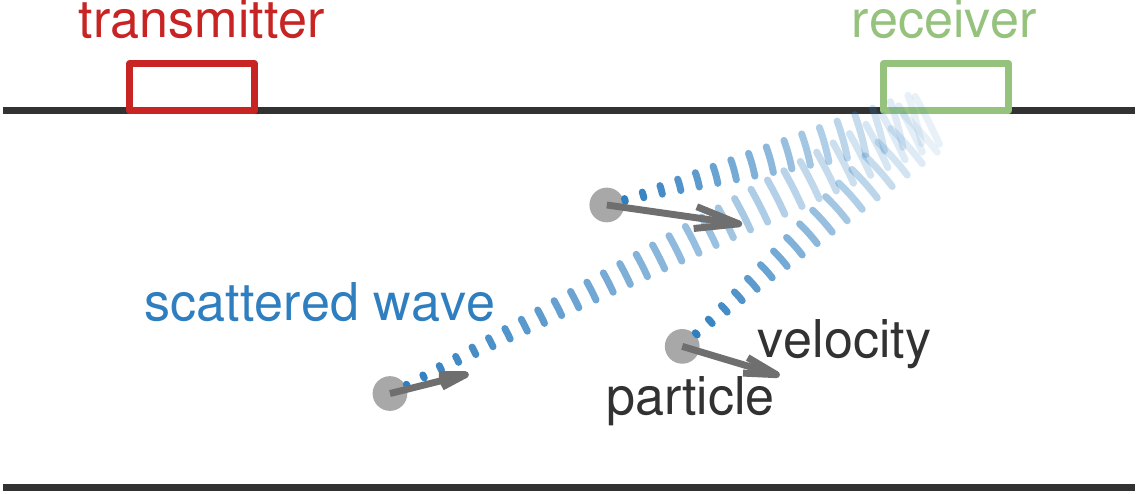}
				\label{subProc2}}
		\end{minipage}
		\caption{The whole physical process of detection with ADCP. (a) The acoustic source wave is emitted by the transmitter and received by the moving particles in the flow during an extremely short period of time. (b) The acoustic wave scattered by the moving particles induces the response of the receivers.
		}
		\label{phyProc}
	\end{figure}
	
	Process 1 in \Fref{subProc1} is to emit the ultrasonic waves from the stationary transmitter \big(the red box in \Fref{subProc1}\big). The inhomogeneous wave equation governs it with constant coefficients and a single stationary transmitter. Then the ultrasonic sound wave is sensed by the moving particles \big(the gray disks in \Fref{subProc1}\big) in the flow field.
	
	Process 2 is crucial to determine the velocity of particles. The moving particles scatter the ultrasonic waves emitted by the transmitter. It is appropriate to think of each particle as a source that emits sound waves uniformly towards the flow in all directions . 
	Then the receiver (the green boxes in \Fref{phyProc}) captures the echoes by measuring the sound wave pressure. Therefore, Process 2 can be viewed as detecting moving acoustic wave sources, i.e., the particles, from data collected on receivers and then determining the velocity of the particles and the flow. 
	However, in practice, it is only possible to detect the particle positions at discrete time points. 
	Since the sound wave speed is much larger than the particle speed, we can reasonably assume that a piecewise constant function can approximate the particle trajectory well. In other words, the particles stay at the same positions between the emitting time and the receiving time of ultrasonic waves. 
	
	To summarize, the original flow measurement problem is recast as an inverse source problem, which consists of two steps:
	\begin{enumerate}[(i)]
		\item Detect the positions of the scattering particles in the flow at each time point from the acoustic wave pressure recorded on receivers.
		\item Determine the moving trajectory of the particles to obtain the particle velocity, i.e., the flow velocity. 
	\end{enumerate} 
	To the best of our knowledge, this modeling strategy is revolutionary, and the numerical simulations demonstrate its high accuracy and efficiency over existing methods.
	
	\section{Direct and inverse source problems for the wave equations\label{sec:directInv}}
	\subsection{Mathematical model of the flow measurement}
	We begin with the inhomogeneous wave equation governing the direct scattering model with moving sources. The sound speed in the flow is denoted by the constant $c$ as before. Let $f_0(x)$ be the initial density of particles. The map $T_t:\mathbb{R}^n\rightarrow \mathbb{R}^n(t\geq 0,n=2,3)$ represents the moving trajectory of particles, meaning that the particle at $x$ is carried to the image $T_t(x)$ by time $t$ under the influence of the flow. Assume that the map $T_t$ is bijective and satisfies $T_0(x)=x$.
	
	Next, let us introduce the definition of push-forward. Let $f_t(x):=(T_t)_\sharp f_0(x)$ be the density of particles at time $t$ where the image density $f_t(x)$ satisfies that for any measurable set $A\subset\mathbb{R}^n$,
	\begin{equation}
		\int_A f_t(x)\mathrm{d}x=\int_{T_t^{-1}(A)} f_0(x)\mathrm{d}x,
	\end{equation}
	i.e., the mass conservation of particles.
	In addition, suppose that the support of $f_t(\cdot)$ is compact in $\Omega\subset \mathbb{R}^n$, a simply connected bounded domain with the Lipschitz boundary, which represents the motion area for the moving particles during the time interval of interests. 
	
	Push-forward ensures that the map $(T_t)_\sharp$ between $f_0$ and $f_t$ is mass-preserving, which means that $f_t(x)$ is also a density function. Hence it is a suitable tool to describe the particle movement in the flow. Then, the particle velocity, representing the flow velocity, at point $x$ and time $t$ is given by $\partial_t T_t(x)$.
	
	For the general distribution of particles in the flow, the sound pressure $U(x,t)$ satisfies the following inhomogeneous wave equation
	\begin{subequations}\label{oriWaveEqns}
		\begin{align}
			&\frac{1}{c^2}\partial^2_t U(x,t)-\Delta U(x,t)=\lambda(x,t)f_t(x),~(x,t)\in\mathbb{R}^n\times \mathbb{R}^+,\label{oriWave}\\
			&U(x,0)=\partial_t U(x,0)=0,~x\in\mathbb{R}^n,\label{zeroIC}
		\end{align}
	\end{subequations}
	where $\lambda(x,t)$ denotes the known sound pressure of the ultrasonic wave emitted by the transmitter in the flow field and $\lambda(x,t)\equiv 0$ for $t\leq 0$.

	
	As explained in section \ref{SecReform}, the sound speed $c$ is much larger than the flow speed $|\partial_t T_t(x)|$, and the particles are almost stationary between the emission and capture of the ultrasonic waves. Let $\{T_j\}$ be a uniform partition with time step $\Delta T$ such that $T_j=j\Delta T, j=0,1,\dots$, and denote the time points for the detection of ultrasonic waves. We introduce a piecewise constant approximation to $f_t(x)$ such that in every small time interval $[T_{j-1},T_j]$, we have
	\begin{equation}
		f_t(x)\approx f_j(x):=f_{T_j}(x),~t\in[T_{j-1},T_j].
	\end{equation}
	Suppose that the ultrasonic wave in previous detection has no impact on the current detection, which means $U(x,T_{j-1})=\partial_t U(x,T_{j-1})=0$ during each detection interval $[T_{j-1},T_j]$. 
	The dynamics of the ultrasonic wave pressure are governed by the inhomogeneous wave equation below: for $j = 1,\ldots$,
	\begin{subequations}\label{pdepiecewise}
		\begin{align}
			& \frac{1}{c^2}\partial^2_t U(x,t)-\Delta U(x,t)=\lambda(x,t)f_{T_j}(x),~(x,t)\in\mathbb{R}^n\times (T_{j-1},T_j],\label{pdepiecewise_a}\\
			&U(x,T_{j-1})=\partial_tU(x,T_{j-1})=0,~x\in\mathbb{R}^n.\label{pdepiecewise_b}
		\end{align}
	\end{subequations}
	Then we can reduce \eqref{pdepiecewise} over each time step to an identical form in a reference time interval $[0,T]$:
	\begin{subequations}\label{appWave_sys}
		\begin{align}
			&\frac{1}{c^2}\partial^2_t U(x,t)-\Delta U(x,t)=\lambda(x,t)f(x),~(x,t)\in\mathbb{R}^n\times (0,T],\label{appWave}\\
			&U(x,0)=\partial_t U(x,0)=0,~x\in\mathbb{R}^n,\label{initCond}
		\end{align}
	\end{subequations}
	where $f(x)$ is the distribution of the stationary particles in $[0,T]$.
	
	Suppose that $\Gamma\subset\partial\Omega$ is a set used for data collection. $\Gamma$ is regarded as a surface in modeling, while in numerical tests, we are interested in the case where $\Gamma$ is a set of discrete points to place the receivers, i.e., $\Gamma=\{r_1,r_2,\dots,r_N\}\subset\partial\Omega$.
	The following proposition shows that the approximation \eqref{appWave_sys} above is proper. For the sake of simplicity, the proposition is restricted to the three-dimensional case, i.e., $n=3$.
	\begin{proposition}
		Assume that $f_t(x)\in C^1(\Omega\times[0,T])$ and $T_t(x)\in C^2(\Omega\times[0,T])$. Let $U_1(x,t),U_2(x,t)$ be the solutions of \eqref{oriWaveEqns} and \eqref{appWave_sys}
		respectively. Then $U_1(x,t),U_2(x,t)$ are sufficiently close provided that $T$ is small enough, namely for any $\epsilon>0$, there exists $\delta>0$ such that for every $0<T<\delta$, we have
		\begin{equation*}
			\|U_1-U_2\|_{L^1(\Gamma\times[0,T])}\leq\epsilon.
		\end{equation*}
	\end{proposition}
	\begin{proof}
		$U_1(x,t)$ and $U_2(x,t)$ can be represented by utilizing the Green's function,
		\begin{equation*}
			U_1(x,t)=\int_0^T\int_{\mathbb{R}^3}G_3(x,t;y,s)\lambda(y,s)f_s(y)\mathrm{d}y\mathrm{d}s,
		\end{equation*}
		\begin{equation*}
			U_2(x,t)=\int_0^T\int_{\mathbb{R}^3}G_3(x,t;y,s)\lambda(y,s)f(y)\mathrm{d}y\mathrm{d}s,
		\end{equation*}
		where $G_3(x,t;y,s)$ is the Green's function associated with the wave operator $\square :=\partial_t^2/c^2-\Delta$ in $\mathbb{R}^3$, i.e.
		\begin{equation*}
			G_3(x,t;y,s)=\frac{\delta(t-s-c^{-1}\|x-y\|)}{4\pi\|x-y\|}.
		\end{equation*}
		Let $m=\inf_{x\in\Gamma,y\in\Omega}\|x-y\|>0$. Then it follows from $\lambda(y,t)\in C^0(\Omega\times[0,T])$ that
		\begin{align}
			|U_1(x,t)-U_2(x,t)|&=\left|\int_{B(x,ct)}\frac{\lambda(x,t-c^{-1}\|x-y\|)}{4\pi\|x-y\|}f_{t-c^{-1}\|x-y\|}(y)-f(y)\mathrm{d}y\right|\notag\\
			&\leq \frac{1}{4\pi m}\|\lambda\|_0\int_{B(x,ct)}|f_{t-c^{-1}\|x-y\|}(y)-f(y)|\mathrm{d}y,\notag
		\end{align}
		where $\|\cdot\|_k$ denotes the norm in $C^k(\Omega\times[0,T])$ and $B(x,r)=\{y\in\mathbb{R}^3,\|y-x\|\leq r\}$. For $s\in[0,T]$,
		\begin{align}
			|f_s(y)-f(y)|&=\big|f_s(y)-|\det\nabla T_s(y)|f_s(T_s(y))\big|\notag\\
			&\leq\big|f_s(y)-f_s(T_s(y))\big|+\big|f_s(T_s(y))\big|\cdot\big|1-|\det\nabla T_s(y)|\big|\notag\\
			&\leq s\|\nabla f_s\|_0\|\partial_s T_s\|_0+s\|f_s\|_0|\partial_s (\det \nabla T_\tau(y))|\notag\\
			&\leq s\|\nabla f_s\|_0\|\partial_s T_s\|_0+18s\|f_s\|_0\|\partial_s \nabla T_s\|_0 \|\nabla T_s\|^2_0\notag\\
			&\leq s\cdot 18\|f_s\|_1\max\{\|T_s\|_1,\|T_s\|^3_2\}.\notag
		\end{align} 
		Let 
		\begin{equation*}
			C_0=18\|f_s\|_1\max\{\|T_s\|_1,\|T_s\|^3_2\},~\delta=\left(\frac{\epsilon}{\frac{C_0}{16 m}\|\lambda\|_0 H(\Gamma)c^2}\right)^{\frac{1}{4}}
		\end{equation*}
		where $H(A)$ is the Hausdorff measure of $A$. Hence, for $T<\delta$,
		
		\begin{align}
			\|U_1-U_2\|_{L^1(\Gamma\times[0,T])}&=\int_0^T\int_\Gamma |U_1(x,t)-U_2(x,t)|\mathrm{d}x\mathrm{d}t\notag\\
			&\leq \frac{1}{4\pi m}\|\lambda\|_0\int_0^T\int_\Gamma\int_{B(x,ct)}\left|f_{t-c^{-1}\|x-y\|}(y)-f(y)\right|\mathrm{d}y\mathrm{d}x\mathrm{d}t\notag\\
			&\leq \frac{C_0}{4\pi m}\|\lambda\|_0 \int_0^T\int_\Gamma\int_{B(x,ct)} \left(t-c^{-1}\|x-y\|\right)\mathrm{d}y\mathrm{d}x\mathrm{d}t\notag\\
			&\leq \frac{C_0}{4\pi m}\|\lambda\|_0 \int_0^T\int_\Gamma\int_{B(x,ct)} t\mathrm{d}y\mathrm{d}x\mathrm{d}t\notag\\
			&= \frac{C_0}{16 m}\|\lambda\|_0 H(\Gamma)c^2T^4\notag\\
			&\leq \epsilon.\notag
		\end{align}
		This completes the proof.
	\end{proof}
	

	\subsection{Direct source problems for the wave equations}
	Note that the system \eqref{appWave_sys} is an inhomogeneous wave equation with the separable source term.
	The direct source problem for the wave equations can be described as following.
	
	{\it{\textbf{The direct source problem for the flow measurement:}
			
			Given the density distribution $f(x)$ of particles, one can obtain the measurement data of the sound pressure at the positions of the receivers, denoted by $U(x,t)|_{\Gamma\times (0,T]}$.}} 
	
	We introduce the forward operator $\mathcal{F}$ to describe the forward modeling such that
	\begin{equation}
		\mathcal{F}:f(x)\mapsto U(x,t)|_{\Gamma\times (0,T]}.
	\end{equation}
	One can show that $\mathcal{F}$ is a linear bounded operator from $L^2(\Omega)$ to $L^2(\Gamma\times[0,T])$ by regularity estimate for the second-order hyperbolic equations; see \cite{evansPartialDifferentialEquations2010} for details.

	\subsection{Inverse source problems for the wave equations}
	In this section, we establish the inverse source problems with respect to the system \eqref{appWave_sys}.  The inverse source problem is stated as follows.
	
	{\it\textbf{The inverse source problem for the flow measurement:}
		
		Given the sound pressure data $U_\mathrm{data} = U(x,t)|_{\Gamma\times (0,T]}$ received on  the boundary, reconstruct the source term $f(x)$ in \eqref{appWave_sys}, i.e., find $f(x)$ such that $\mathcal{F}f=U_\mathrm{data}$.
	}
	
	\begin{remark}
		After modeling using wave equations, flow measurement and \linebreak PAT/TAT show striking similarities. They are both problems of recovering the source term \big($f_t(x)\lambda(x,t)$ for flow measurement and $f(x)\mathrm{d} \delta(t)/\mathrm{d} t$ for PAT/TAT\big) from the known data received by acoustic pressure transducers on a surface $\Gamma$ based on propagating process of the ultrasonic sound wave.
	\end{remark}
	
	It is well-known that the uniqueness of recovering a general source term $F(x,t)$ from boundary data does not hold. For instance, let $u\in C_0^\infty(\Omega\times[0,T])$ and $F(x,t)=\partial^2_t u/c^2-\Delta u$. Then the boundary data $u|_{\Gamma\times (0,T]}$, namely the observation, is always equal to zero. This means that the received data do not allow the uniqueness of reconstruction of the non-trivial source term. We refer to \cite{isakovInverseSourceProblems1990,dehoopUniquenessSeismicInverse2016a,baoInverseSourceProblems2018} for the uniqueness and stability of various inverse source problems.
	
	In the context of flow measurement, the unknown source term can be decomposed into a product of the known part $\lambda(x,t)$, representing the background wavefield, and the unknown part $f(x)$, representing the particle positions. In the following, we provide a uniqueness result of the inverse source problems with a specific type of unknown sources, i.e. $\lambda(x,t)=h(ct-p\cdot x)$, which is known as traveling plane waves. Without loss of generality, we suppose that $\Omega$ is compactly supported in $ B(0,1)$ for the following proposition.
	
	
	\begin{proposition}\label{thm:uniqueness}
		Assume that $h(\tau)=0,\tau<1$ and $\hat{h}(\zeta)$ is nonvanishing for every $\zeta$, where 
		\begin{equation*}
			\hat{h}(\zeta)=\int_{-\infty}^\infty h(\tau)e^{i\zeta \tau}\mathrm{d}\tau.
		\end{equation*}
		Suppose that the measurement surface $\Gamma=\partial\Omega$. The reference time $T$ is chosen properly to ensure that the scattered wave can be entirely received. Then $\mathcal{F}$ is injective.
	\end{proposition}
	\begin{proof}
		Note that the operator $\mathcal{F}$ is linear in $L^2(\Omega)$. Then it is sufficient to consider the case of $U_\mathrm{data}=0$, i.e. $U(x,t)|_{\Gamma\times[0,T]}=0$. For the choice of $T$, the solution $U(x,t)$ in $\Omega$ vanishes when $t\geq T$, which implies $U(x,t)=\partial _t U(x,t)=0,x\in\Omega,t\geq T$. Then multiplying both sides of \eqref{appWave} by the test function $\varphi(x,t)=e^{i(\omega t+d\cdot x)}$ with $d\in\mathbb{R}^n,|d|=\omega/c$ and integrating by parts over $\Omega\times[0,\infty)$ yield
		\begin{equation}
			\int_{\Omega\times[0,\infty)}\lambda(x,t)f(x)\varphi(x,t)\mathrm{d}x\mathrm{d}t=\int_0^\infty\int_{\Gamma}\big(U(x,t)\partial_\nu\varphi(x,t)-\partial_\nu U(x,t)\varphi(x,t)\big)\mathrm{d}S\mathrm{d}t.
			\label{propUniqueIBP}
		\end{equation}
		Since $U$ is a solution to the homogeneous wave equation in $\big(\mathbb{R}^n\setminus \bar{\Omega}\big)\times[0,\infty)$ and $U(x,t)=0$ on $\Gamma\times[0,\infty)$, then from known results we obtain that $U\equiv0$ in $\big(\mathbb{R}^n\setminus \bar{\Omega}\big)\times[0,T]$. Hence $\partial_\nu U(x,t)=0$ on $\Gamma\times[0,\infty)$, together with $h(z)\equiv0,z<1$, implies that
		\begin{equation*}
			\int_{\Omega\times[0,\infty)}h(ct-p\cdot x)f(x)e^{i(\omega t+d\cdot x)}\mathrm{d}x\mathrm{d}t=0,~|d|=\omega/c.
		\end{equation*}
		Equivalently,
		\begin{align}
			&\int_{\Omega}\frac{1}{c}e^{i\frac{\omega}{c}p\cdot x}\hat{h}(\frac{\omega}{c})f(x)e^{id\cdot x}\mathrm{d}x=0,~|d|=\omega/c.\notag\\
			\Longleftrightarrow& \hat{h}(|d|)\int_{\Omega}f(x)e^{i(d+|d|p)\cdot x}\mathrm{d}x=0.\notag\\
			\Longleftrightarrow& \hat{h}(|d|)\hat{f}(d+|d|p)=0.\notag
		\end{align}
		Hence $\hat{f}(d+|d|p)=0$ from $\hat{h}(\zeta)\not=0$, which is in fact that $\hat{f}(\xi)=0$ for $\xi\in\{x\in\mathbb{R}^n,x\cdot p>0\}\cup \{0\}$. Considering that $\hat{f}(-\xi)=\overline{\hat{f}(\xi)}$, we obtain $\hat{f}(\xi)=0$ a.e. in $\mathbb{R}^n$. Finally from the invertibility of the Fourier transform we get $f(x)=0$, which completes the proof.
	\end{proof}

	\begin{remark}
		We think that one can obtain a uniqueness result for the case when $\partial\Omega\setminus\Gamma$ is a part of a plane or of a sphere using the argument by \cite{isakovUniquenessInverseConductivity2007}. This type of boundary enables to reflect the solution across $\partial\Omega\setminus\Gamma$. In addition, it should be noted that the stability is another significant property of the inverse source problem and is a component in the definition of the well-posedness. The stability of an inverse problem means that one can get the almost exact solution even if the observation data is noisy. The stability also guarantees the convergence rate of the iterative method; see \cite{dehoopLocalAnalysisInverse2012,mittalConvergenceRatesIteratively2022,dehoopAnalysisMultilevelProjected2015a,mittalIterativelyRegularizedLandweber2021,wangConvergenceAnalysisInexact2018} for details. We refer to \cite{baoInverseSourceProblem2011,kirkebyStabilityInverseSource2020,chengIncreasingStabilityInverse2016,liIncreasingStabilityInverse2017} for investigations of the stability of various inverse source problems for wave equations in the frequency domain.
	\end{remark}

	While the full boundary measurements are ideal, only partial data are available in practice since the receivers can only be arranged on the surface of the flow, above the riverbed, or at specific positions within the water flow. In short, the distribution of transducers featuring data collection is quite sparse in the flow. Moreover, the layouts of receivers can be modified in our model according to different application scenarios.  
	
	Proposition \ref{thm:uniqueness} guarantees that the measurement data on $\Gamma=\partial\Omega$ can uniquely determine the particle positions in principle. However, it is challenging to directly solve the operator equation with respect to $f$. We propose the corresponding minimization problem:
	\begin{equation}
		\mathcal{J}(f):=\frac{1}{2}\|\mathcal{F}(f)-U_\mathrm{data}\|^2_{L^2(\Gamma\times (0,T])}.
		\label{funcJ}
	\end{equation}
	According to Proposition \ref{thm:uniqueness}, the global minimizer exists and is unique.
	
	\section{Numerical schemes and experiments \label{sec:numer results}}
	In this section, we will first explain the approach to minimize the objective functional \eqref{funcJ} and then present several examples to demonstrate the performance of our mathematical model for the flow measurement.
	The density distribution $f(x)$ of particles is of the form:
	\[
	f(x) = \sum_{m=1}^M \chi_{\{2|x-s_m|\leq D\}}(x),
	\]
	where $M$ particles with uniform diameter $D$ m are located at center $s_m$ for $m = 1, \ldots, M$, respectively, and $\chi$ is the characteristic function.
	\subsection{Statement of the minimization problem}\label{subsec:statament of the minimization problem}
	We minimize the objective functional $\mathcal{J}(f)$ defined by \eqref{funcJ}, via gradient typed method. Various iterative methods can be applied for solving the typical linear least squares problems. In particular, the conjugate gradient method has been shown to be robust and numerically efficient when dealing with a wide range of optimization problems. Our numerical experiments demonstrate that the conjugate gradient method is also efficient for the inverse source problems for wave equations.

	Iterative methods for minimizing the functional $\mathcal{J}(f)$ require the knowledge of the first-order derivative. In the following proposition, we derive the expressions of the first-order Fr\'{e}chet derivative of the functional $\mathcal{J}(f)$ and the adjoint operator $\mathcal{F}^*:L^2(\Gamma\times(0,T])\rightarrow L^2(\Omega)$, which are employed for iterative algorithms. More precisely, we have the following result.
	\begin{proposition}
		The first-order Fr\'{e}chet derivative of functional $\mathcal{J}(f)$ is given by
		\begin{equation}
			D\mathcal{J}(f)=\mathcal{F}^*\big(\mathcal{F}(f)-U_\mathrm{data}\big).
			\label{gradF}
		\end{equation}
		For the adjoint operator
		\[\mathcal{F}^*:V(x,t)\mapsto g(x),\]
		with $V(x,t)\in L^2(\Gamma\times(0,T]), g(x)\in L^2(\Omega)$, we consider the time reversed inhomogeneous wave equations with constant coefficients,
		\begin{subequations}
			\begin{align}
				&\frac{1}{c^2}\partial^2_t W(x,t)-\Delta W(x,t)=V(x,t),~(x,t)\in\mathbb{R}^n\times (0,T],\label{adjointAppWave}\\
				&W(x,T)=\partial_t W(x,T)=0,~x\in\mathbb{R}^n.\label{adjointInitCond}
			\end{align}
		\end{subequations}
		Then the image $g(x)$ is given by
		\begin{equation}
			g(x)=\int_0^T \lambda(x,t)W(x,t)\mathrm{d}t.
			\label{imAdjoint}
		\end{equation}
	\end{proposition}
	\begin{proof}
		By the definition of Fr\'{e}chet derivative we get
		\begin{equation*}
			D\mathcal{J}(f)=\mathcal{F}^*\big(\mathcal{F}(f)-U_\text{data}\big).
		\end{equation*}
		
		Assume that $V(x,t)\in L^2(\Gamma\times(0,T])$. By the definition of the adjoint operator and the fact that the support of $V(\cdot,t)$ is contained in $\Gamma$, we have
		\begin{align}
			\langle f,\mathcal{F}^*(V)\rangle_{L^2(\Omega)}&=\langle \mathcal{F}(f),V\rangle_{L^2(\mathbb{R}^n\times(0,T])}\notag\\
			&=\langle U_\text{data},V\rangle_{L^2(\mathbb{R}^n\times(0,T])}\notag\\
			&=\langle U,\frac{1}{c^2}\partial^2_t W(x,t)-\Delta W(x,t)\rangle_{L^2(\mathbb{R}^n\times(0,T])}.\notag
		\end{align}
		Then note that the compact supports of solutions $U(\cdot,t)$ and $W(\cdot,t)$ in $\mathbb{R}^n$, the initial conditions of $U(x,t)$ and the final conditions of $W(x,t)$, then we obtain by integration by parts,
		\begin{align}
			\langle f,\mathcal{F}^*(V)\rangle_{L^2(\Omega)}&=\langle \frac{1}{c^2}\partial^2_t U(x,t)-\Delta U(x,t),W(x,t)\rangle_{L^2(\mathbb{R}^n\times(0,T])}\notag\\
			&=\langle \lambda(x,t)f(x),W(x,t)\rangle_{L^2(\mathbb{R}^n\times(0,T])}\notag\\
			&=\langle f(x),\int_0^T \lambda(x,t)W(x,t)\mathrm{d}t\rangle_{L^2(\mathbb{R}^n)}\notag\\
			&=\langle f(x),\int_0^T \lambda(x,t)W(x,t)\mathrm{d}t\rangle_{L^2(\Omega)},\notag
		\end{align}
		which implies that
		\begin{equation}
			g(x)=\mathcal{F}^*(V)=\int_0^T \lambda(x,t)W(x,t)\mathrm{d}t.\notag
		\end{equation}
		This completes the proof.
	\end{proof}
	\subsection{Reconstruction of the stationary source term}
	In this part, we test the validity and accuracy of the inverse source problems for the wave equations with the stationary source term corresponding to \eqref{appWave_sys} in $\mathbb{R}^2$ for different parameters. 
	
	
	For all examples in this subsection, we take $472\times216$ grid points, uniformly distributed in the domain $\Omega=[0,L_1]\times[0,L_2]=[0,4.71~\mathrm{m}]\times[0,2.15~\mathrm{m}]$. The sound speed is chosen as $c=1500 ~\mathrm{m/s}$. We arrange 10 circular particles as stationary source terms in $\Omega$. The diameter of particles in the flow is chosen to be $0.14$ m according to the statistical data in \cite{stoneEvaluatingVelocityMeasurement2007} where median particle diameters were 0.109 m and 0.114 m for the St. Maries and Potlatch rivers. The time increment is selected based on Courant–Friedrichs–Lewy condition for wave equations.
	We simulate the propagation of the ultrasonic wave in the homogeneous medium governed by \eqref{appWave_sys} in MATLAB with the k-Wave toolbox\cite{treebyKWaveMATLABToolbox2010}.
	
	In the following numerical experiments, we present several settings when choosing different parameters, such as the frequency of the source wave, the layout of the receivers, and the level of noise in the registered data, to verify the validity of different parameter choices. These experiments can be used to perform flow measurements in real-world environments.
	
	\subsubsection{Example 1: Different frequencies of source wave\label{subsubsec:eg1}}
	In this example, we assume that $\lambda(x,t)$ is independent of $x$ and take the Gaussian signal as the emitted source wave by circular particles, which is given by
	\begin{equation}
		\lambda(x,t)=\exp \left(-\pi^{2} q_{0}^{2}\left(t-\frac{p}{2}\right)^{2}\right), ~ t \in[0, p],
	\end{equation}
	where $q_0$ is the central frequency of Gaussian signal and $[0,p]$ is the support of Gaussian signal. The parameter $p$ varies with the central frequency as $p=6/(\pi q_0)$.
	This experiment aims to test the suitable frequency range with which the position of sources can be determined precisely.
	
	The received data are presented in \Fref{recDiffFreq}. \Fref{diffFreq} shows the contours of the density functions $f(x)$.
	\Fref{diffFreqA} shows the exact positions of particles by yellow disks.
	The reconstructions of the source term are shown in \Fref{diffFreq}(b-d) where the red points represent the locations of receivers. The relative $L^2$ errors of reconstructions for each case are given in \Tref{errorEg1}.
	\begin{table}
		\centering
		\caption{Relative $L^2$ errors of reconstructions after 100 iterations.}
		\begin{tabular}{*{4}{c}}
			\toprule
			Frequency (kHz)&1&10 &100 \\\midrule
			Relative $L^2$ error &$9.103\times 10^{-1}$&$2.763\times10^{-1}$&$1.713\times 10^{-6}$\\
			\bottomrule
		\end{tabular}
		\label{errorEg1}
	\end{table}
	
	One can observe that the image of the registered data is getting sharper as the central frequency increases, resulting in the different resolution of recovering source terms. 
	The reconstruction is not able to capture the positions of sources when the frequency of the source wave is less than 1 kHz; see \Fref{diffFreqB}. 
	Although the relative $L^2$ error is 27.63\%, which seems large, the recovering image for 10 kHz still allows for good determination of the positions of particles; see \Fref{diffFreqC}. When the frequency of the source wave is sufficiently large, the reconstruction of the position of sources is of satisfactory quality; see \Fref{diffFreqD}. Therefore, a source wave with a central frequency of more than 10 kHz is sufficient to detect particles of 0.14 m diameter.
	
	\begin{remark}
		The numerical results are consistent with the strategy of	choosing suitable central frequency according to the size of particles. In general, the spatial resolution limit is regarded as being one-half of the wavelength. In the above setting, it is straightforward to determine the minimum central frequency for imaging,
		\begin{equation*}
			q_{0,\text{min}}\approx \frac{1}{2}\cdot\frac{1500 \text{~m/s}}{0.14 \text{~m}}\approx 5.357 \text{~kHz}.
		\end{equation*}
		
		On the other hand, it is well known that higher frequencies produce a higher resolution, but depth of penetration is limited in practical measurements. Thus taking into account these two considerations, we propose to select test frequencies close to $10$ kHz. 
	\end{remark}
	\begin{figure}
		\centering
		\begin{minipage}[t]{0.32\textwidth}
			\centering
			\subfigure[1 kHz]{\includegraphics[width=0.95\columnwidth]{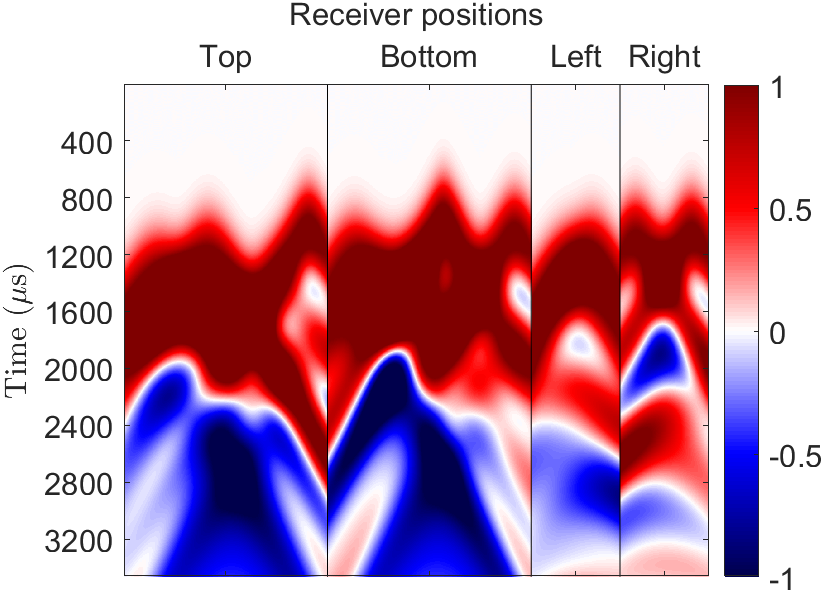}}
		\end{minipage}
		\begin{minipage}[t]{0.32\textwidth}
			\centering
			\subfigure[10 kHz]{\includegraphics[width=0.95\columnwidth]{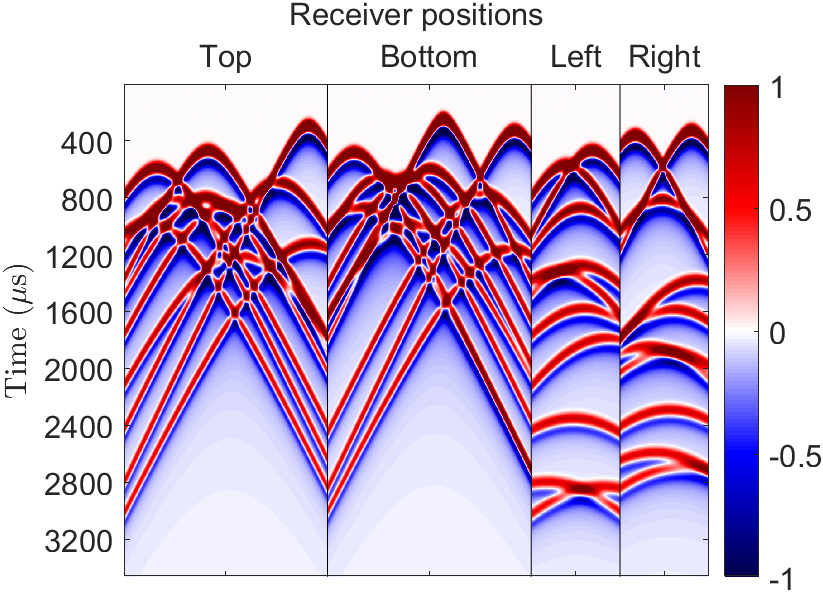}}
		\end{minipage}
		\begin{minipage}[t]{0.32\textwidth}
			\centering
			\subfigure[100 kHz]{\includegraphics[width=0.95\columnwidth]{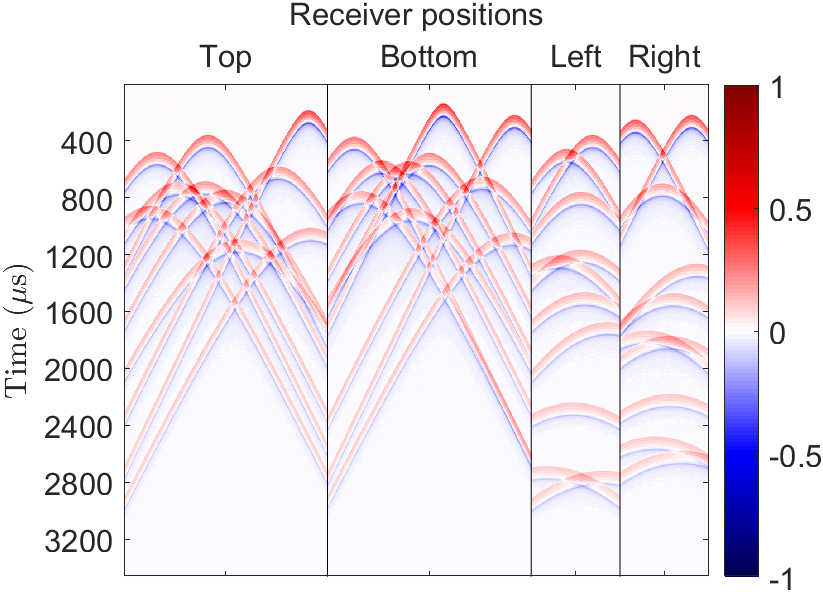}}
		\end{minipage}
		\caption{Received data $U_\mathrm{data}$ in the settings of different central frequencies $q_0$ of the source wave. Colors represent the intensities of received data. The horizontal and vertical axes represent the time step when solving the forward problem and the position of receivers, respectively.}
		\label{recDiffFreq}
	\end{figure}
	\begin{figure}
		\centering
		%
		
		\begin{minipage}[t]{0.45\textwidth}
			\flushright
			\subfigure[Ground truth]{\includegraphics[width=0.67\columnwidth]{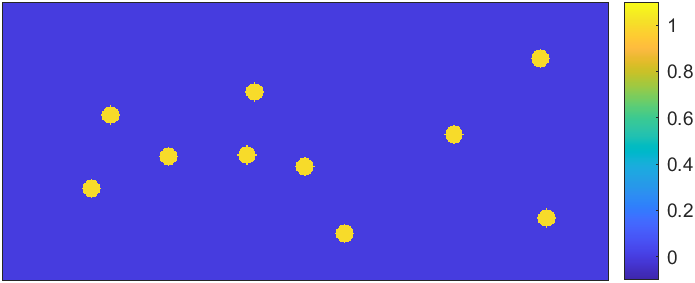}
				\label{diffFreqA}}
		\end{minipage}
		\begin{minipage}[t]{0.45\textwidth}
			\subfigure[1 kHz]{\includegraphics[width=0.67\columnwidth]{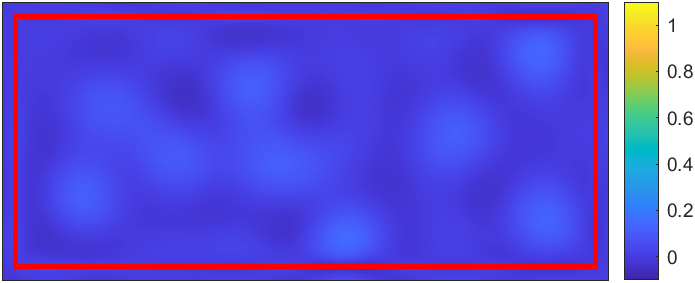}
				\label{diffFreqB}}
		\end{minipage}
		
		\begin{minipage}[t]{0.45\textwidth}
			\flushright
			\subfigure[10 kHz]{\includegraphics[width=0.67\columnwidth]{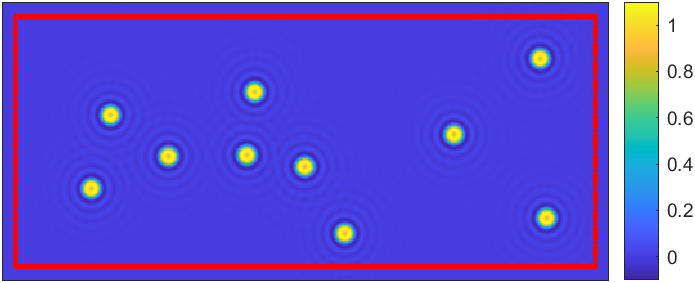}
				\label{diffFreqC}}
		\end{minipage}
		\begin{minipage}[t]{0.45\textwidth}
			\subfigure[100 kHz]{\includegraphics[width=0.67\columnwidth]{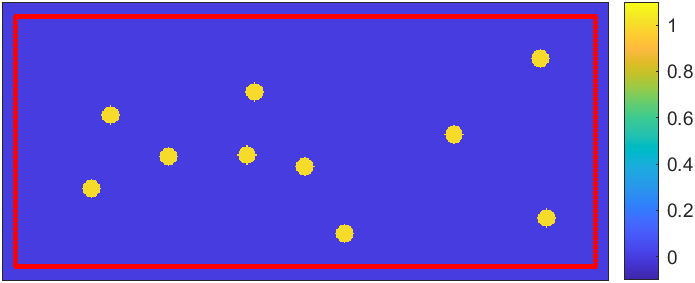}
				\label{diffFreqD}}
		\end{minipage}
		%
		
		%
		\caption{The contours of $f(x)$, i.e., reconstructions of the source term $f(x)$ for different central frequencies $q_0$ of the source wave. The red points represent the receivers. (a) is the exact locations of particles. (b-d) are reconstructions of the source term when the frequency of the source wave is 100 kHz, 10 kHz, and 1 kHz, respectively.}
		\label{diffFreq}
	\end{figure}
	\subsubsection{Example 2: Different layouts of receivers\label{subsubsec:eg2}}
	In this example, the receivers are placed in different positions in order to compare the performance of the model and select the suitable layouts that balance the accuracy and cost. 
	The configuration with more receivers will increase the economic cost. However, fewer receivers would lead to less registered data, which greatly impacts the quality of reconstructions. Therefore, we design a example to tackle the problem of how to balance the expense of receivers and the quality of reconstruction.
	This setting can be meaningful to practical measurement. Suppose that this two-dimensional region $\Omega$ is the longitudinal section of the flow. Then the top and bottom layouts of receivers correspond to the configuration that receivers are close to the surface of the river and the riverbed.
	The lateral layout of receivers represents the configuration of the vertically located receivers in the river.
	
	In this experiment, we take the Gaussian signal as the source wave. The central frequency is set to be $100$ kHz, which shows excellent performances in section \ref{subsubsec:eg1}. The received data are shown in \Fref{recDiffLayout}. \Fref{diffLayout} shows the contours of the density function $f(x)$. \Fref{diffLayoutA} shows the exact positions of particles by yellow disks. \Fref{diffLayout}(b-d) present the reconstructions when setting different layouts of receivers. The relative $L^2$ errors for each case are shown in \Tref{errorEg2}.
	\begin{table}
		\centering
		\caption{Relative errors of reconstructions after 100 iterations.}
		\begin{tabular}{*{4}{c}}
			\toprule
			Layout & Top & Top and bottom & All-around\\\midrule
			Relative $L^2$ error & $3.525\times10^{-1}$ &$1.141\times 10^{-1}$&$1.713\times10^{-6}$\\\bottomrule
		\end{tabular}
		\label{errorEg2}
	\end{table}
	It shows that the inversion with the all-around layout of receivers performs the best of three, as expected. 
	Besides, there are some undesired errors around the sources in the inversions due to the loss detection angles in the other two cases. In the case that data are collected on top receivers in \Fref{diffLayoutB}, we get less information for the particles near the bottom.
	\begin{figure}
		\centering
		\begin{minipage}[t]{0.3\textwidth}
			\centering
			\subfigure{\includegraphics[width=0.95\columnwidth]{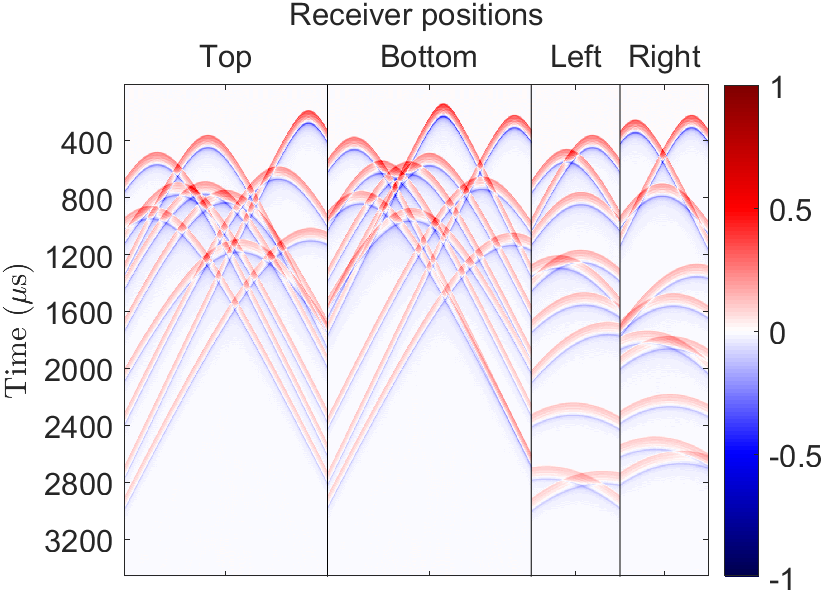}}
		\end{minipage}
		\caption{Received data $U_\mathrm{data}$ in the settings of different layouts of receivers. Colors represent the intensities of received data. The horizontal and vertical axes represent the time step when solving the forward problem and the position of receivers, respectively.}
		\label{recDiffLayout}
	\end{figure}
	\begin{figure}
		\centering
		
		\begin{minipage}[t]{0.45\textwidth}
			\flushright
			\subfigure[Ground truth]{\includegraphics[width=0.67\columnwidth]{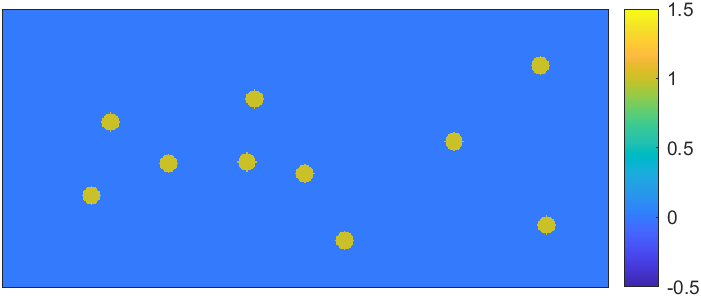}
				\label{diffLayoutA}}
		\end{minipage}
		\begin{minipage}[t]{0.45\textwidth}
			\subfigure[Top layout]{\includegraphics[width=0.67\columnwidth]{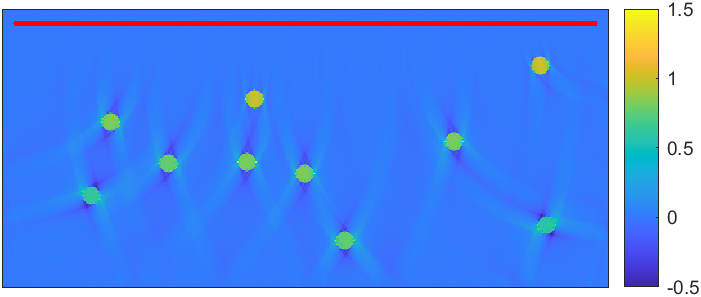}
				\label{diffLayoutB}}
		\end{minipage}
		
		\begin{minipage}[t]{0.45\textwidth}
			\flushright
			\subfigure[Top and bottom layout]{\includegraphics[width=0.67\columnwidth]{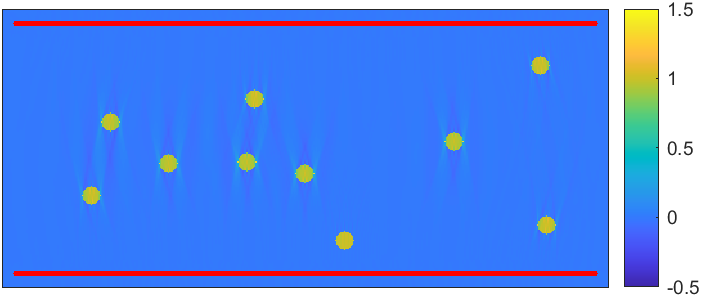}
				\label{diffLayoutC}}
		\end{minipage}
		\begin{minipage}[t]{0.45\textwidth}
			\subfigure[All-around layout]{\includegraphics[width=0.67\columnwidth]{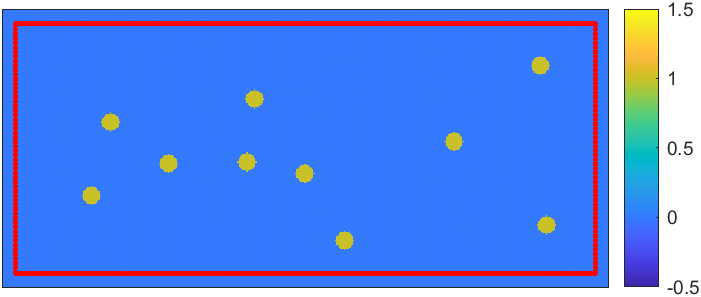}
				\label{diffLayoutD}}
		\end{minipage}
		%
		
		
		\caption{The contours of $f(x)$, i.e., reconstructions of the source term $f(x)$ for different layouts of receivers. The red points represent the receivers. (a) is the ground truth of particles. (b-d) are reconstructions of the source term when choosing different layouts of receivers.}
		\label{diffLayout}
	\end{figure}
	
	\subsubsection{Example 3: Different noise levels}
	The received data in practical measurements are susceptible to contamination from various sources, including instrument noise, an inherent limit to accuracy, and random outliers, which occur due to multiple scatterings by particles. We refer to \cite{durgeshNoiseCorrectionTurbulent2014a,elgarCurrentMeterPerformance2001} on denoising the data. In order to test the robustness of our model in terms of noise, we add noise to received data $U_\text{data}$ and then recover the source term $f(x)$ from noised data. Since the noise of instruments based on the Doppler effect used in the flow measurement, such as ADV and ADCP, is well approximated as unbiased Gaussian white noise (see \cite{rennieDeconvolutionTechniqueSeparate2007,durgeshNoiseCorrectionTurbulent2014a} and references therein for details), we test our model when adding zero-mean Gaussian white noise with different variances. Specifically, contamination data with the zero-mean white noise of variance $\sigma^2$, which indicates the noise level, is given by
	\begin{equation}
		\tilde{U}_\text{data}=U_\text{data}+\max |U_\text{data}|\cdot\sigma \mathcal{N},
	\end{equation}
	where the random variable $\mathcal{N}$ follows the standard normal distribution. Then perturbed data $\tilde{U}_\text{data}$ are used to reconstruct the particle sources $f(x)$.
	
	\Fref{recDiffNoise} contains the received data when adding noise with different levels. \Fref{diffNoise} shows the contours of the density function $f(x)$. \Fref{diffNoiseA}  shows the exact positions of particles by yellow disks. \Fref{diffNoise}(b-d) present the reconstructions when adding noise with different levels. It can be observed that the received data consist of almost noise and little useful information when the noise level is high in \Fref{recDiffNoise}. However, the reconstructions in \Fref{diffNoise} show great performances, where the particles appear clearly. The experiments of our model demonstrate that the reconstructions maintain excellent performances even if the noise level is quite high.
	
	\begin{figure}
		\centering
		\begin{minipage}[t]{0.32\textwidth}
			\centering
			\subfigure[$\sigma=0.05$]{\includegraphics[width=0.95\columnwidth]{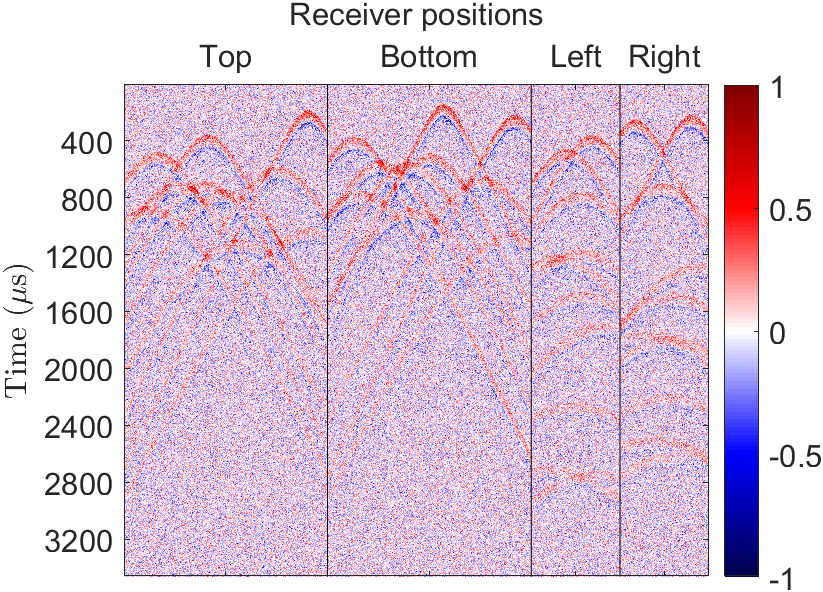}}
		\end{minipage}
		\begin{minipage}[t]{0.32\textwidth}
			\centering
			\subfigure[$\sigma=0.1$]{\includegraphics[width=0.95\columnwidth]{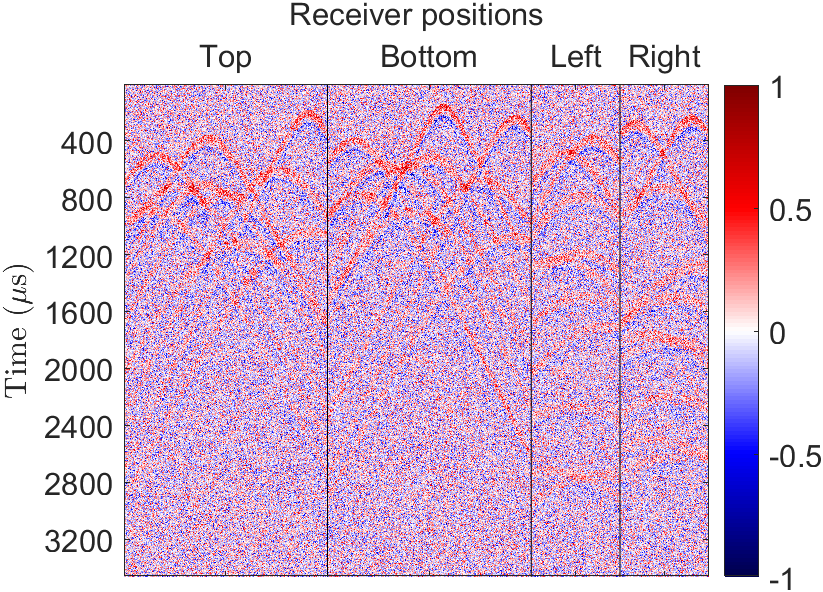}}
		\end{minipage}
		\begin{minipage}[t]{0.32\textwidth}
			\centering
			\subfigure[$\sigma=0.2$]{\includegraphics[width=0.95\columnwidth]{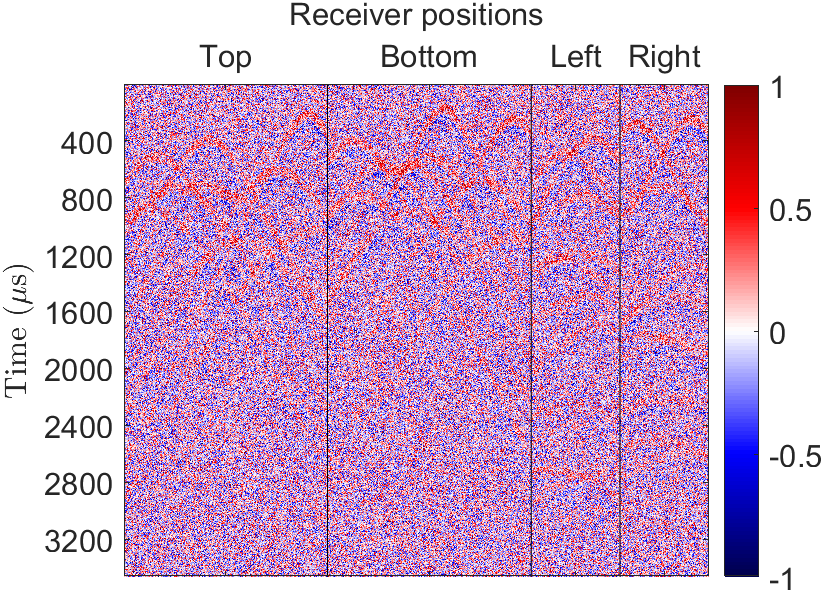}}
		\end{minipage}
		\caption{Received data $U_\mathrm{data}$ for different noise levels $\sigma$. Colors represent intensities of received data. The horizontal axis and vertical axis represent the time step when solving the forward problem and the position of receivers, respectively.}
		\label{recDiffNoise}
	\end{figure}
	\begin{figure}
		\centering
		\begin{minipage}[t]{0.45\textwidth}
			\flushright
			\subfigure[Ground truth]{\includegraphics[width=0.67\columnwidth]{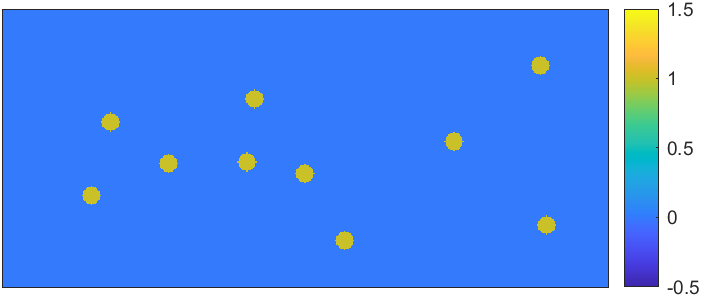}
				\label{diffNoiseA}}
		\end{minipage}
		\begin{minipage}[t]{0.45\textwidth}
			\subfigure[$\sigma=0.05$]{\includegraphics[width=0.67\columnwidth]{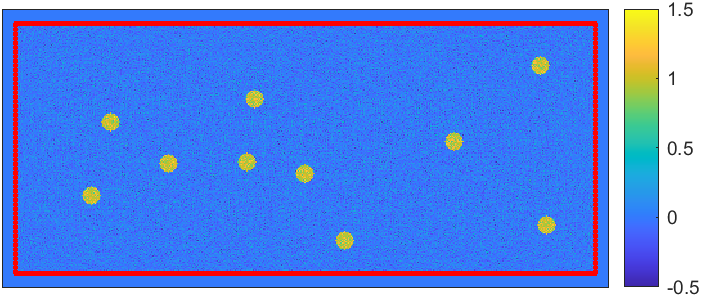}
				\label{diffNoiseB}}
		\end{minipage}
		
		\begin{minipage}[t]{0.45\textwidth}
			\flushright
			\subfigure[$\sigma=0.1$]{\includegraphics[width=0.67\columnwidth]{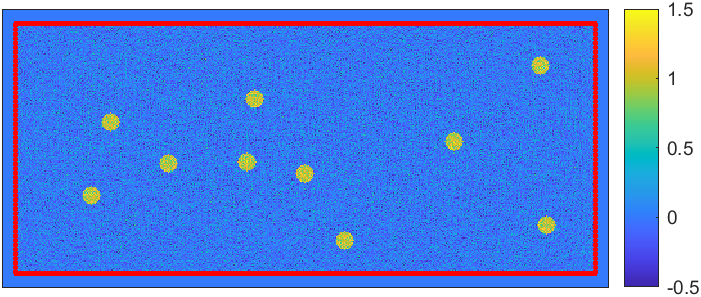}
				\label{diffNoiseC}}
		\end{minipage}
		\begin{minipage}[t]{0.45\textwidth}
			\subfigure[$\sigma=0.2$]{\includegraphics[width=0.67\columnwidth]{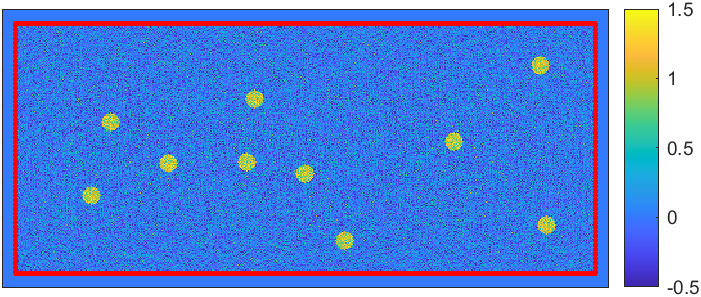}
				\label{diffNoiseD}}
		\end{minipage}
		\caption{The contours of $f(x)$, i.e., reconstructions of the source term $f(x)$ for different noise levels $\sigma$. The red points represent the receivers. (a) is the ground truth of particles. (b-d) are reconstructions of the source term for noise level $\sigma=0.05,~0.1,~0.2$.}
		\label{diffNoise}
	\end{figure}
	
	\subsection{Reconstruction of the moving source term and the flow field}
	In this section, suppose that the positions and sizes of sources are computed through the gradient based method introduced in section \ref{subsec:statament of the minimization problem}.
	We consider the inversion of the velocity field of the synthetic flow according to the model for moving particles given by \eqref{oriWaveEqns}. The whole procedure consists of recovering of particle velocities and reconstruction of the velocity field of the flow.
	The algorithms, such as nearest point searching and optical flow algorithm, are applied for recovering the local (or whole, dependent on distribution and number of sources) flow field.
	
	\subsubsection{Preparation of registering data on receivers}
	Since a synthetic flow is of interest, we prepare the registering data on receivers using numerical results of the direct source problem.
	Given a known time-dependent vector field $\mathbf{r}(x,t)$ in two-dimensional space representing the flow velocity at position $x=(x_1,x_2)$ and time $t$, the problem encountered first is how to calculate the trajectory of a moving particle initially located at point $x$. The trajectory $C(x,t)$ is governed by the following ordinary differential equation,
	\begin{equation}
		\partial_t C(x,t)=\mathbf{r}\big(C(x,t),t\big),~C(x,0)=x.
		\label{trajODE}
	\end{equation}
	In the following numerical example, the forward Euler scheme is applied for solving \eqref{trajODE} to obtain the trajectory of particles, which is given by
	\begin{equation}
		\frac{C(x,t_{j+1})-C(x,t_j)}{\Delta t}=\mathbf{r}\big(C(x,t_j),t_j\big),~C(x,t_0)=x.
	\end{equation}
	
	Once getting the trajectory of every particle, we solve the forward problems \eqref{appWave_sys} periodically. Meanwhile, sound pressure data $U(x,t)|_{\Gamma\times(0,T]}$ are collected on the fixed receivers $r_1, r_2,\dots,r_N$. This whole procedure allows us to simulate the processes of sound wave scattering by moving particles and registering data on receivers.
	
	
	\subsubsection{Inversion for velocity of particles}
	
	When the distribution of particles in the flow is relatively sparse and the displacement of each particle is relatively small due to the short period between adjacent time sections, it is reasonable to use the method of nearest point searching to determine the particle positions and particle velocities.  
	
	For the case of a large number of particles, identifying the corresponding particles in adjacent time sections is quite challenging. Thus another approach, the optical flow algorithm, is employed to deal with this situation. The optical flow is defined as the apparent motion of image objects between two consecutive frames, which is a fundamental approach to calculating the motion of image intensities. We refer the readers to \cite{broxHighAccuracyOptical2004,bruhnLucasKanadeMeets2005} for a comprehensive survey of optical flow techniques. In the case of a large number of particles, all reconstructed frames of moving particles could also be regarded as a sequence of intensity images. In addition, the minor difference between consecutive frames also makes the optical flow approach an appropriate choice to calculate the flow field. The implementation of optical flow techniques used in the following examples is described in \cite{liuPixelsExploringNew2009,broxHighAccuracyOptical2004,bruhnLucasKanadeMeets2005}.
	
	For the sake of testifying the accuracy of recovering the velocity of flow at multiple points and reconstructing the whole flow field, we design two numerical experiments. The main difference between them is the number and the size of particles in the flow. The configuration with a small number of large-sized particles aims to study the accuracy of recovering the local velocity of flow. Another configuration with a large number of small-sized particles aims to test the accuracy of rebuilding the velocity field of the whole flow. The results of numerical experiments are shown as follows.
	
	%
	%
	
	\subsubsection{Inversion for the Taylor-Green vortex}
	We choose the Taylor-Green vortex as the background flow field in this experiment. The Taylor-Green vortex was first studied in \cite{taylorMechanismProductionSmall1937}, which has a closed form of solution to incompressible Navier-Stokes equations in the two-dimensional space. The time-independent Taylor-Green vortex in the domain $\Omega=[0,L_1]\times[0,L_2]$ is given by
	\begin{equation}
		\mathbf{r}(x,t)=\left(\cos \left(\frac{2\pi}{L_1}x_1\right)\sin \left(\frac{2\pi}{L_2}x_2\right),-\sin\left( \frac{2\pi}{L_1}x_1\right)\cos \left(\frac{2\pi}{L_2}x_2\right)\right).
	\end{equation} 
	
	We place 13 particles with a diameter of 0.14 m in the background flow. The initial positions of particles and several time sections in $[0,T]$ are presented in the left column of \Fref{invFew}. The figures in the middle and right columns present the inversions of particles using the iterative method in section \ref{sec:directInv}. The white and green arrows indicate the exact velocity vectors and the approximated velocity vectors calculated from the inversions. The approximated vectors in the middle and right columns are obtained using the nearest point searching and optical flow methods, respectively. 
	
	As can be seen, there is a good agreement between the computed velocity vectors and the ground truth. 
	The relative error of the nearest point searching and optical flow method are $4.61\%$ and $3.75\%$ respectively, demonstrating the reconstruction is a pretty excellent approximation of the original flow field.
	
	\begin{figure}
		\centering
		\begin{minipage}[t]{0.3\textwidth}
			\subfigure{
				\includegraphics[width=0.95\columnwidth]{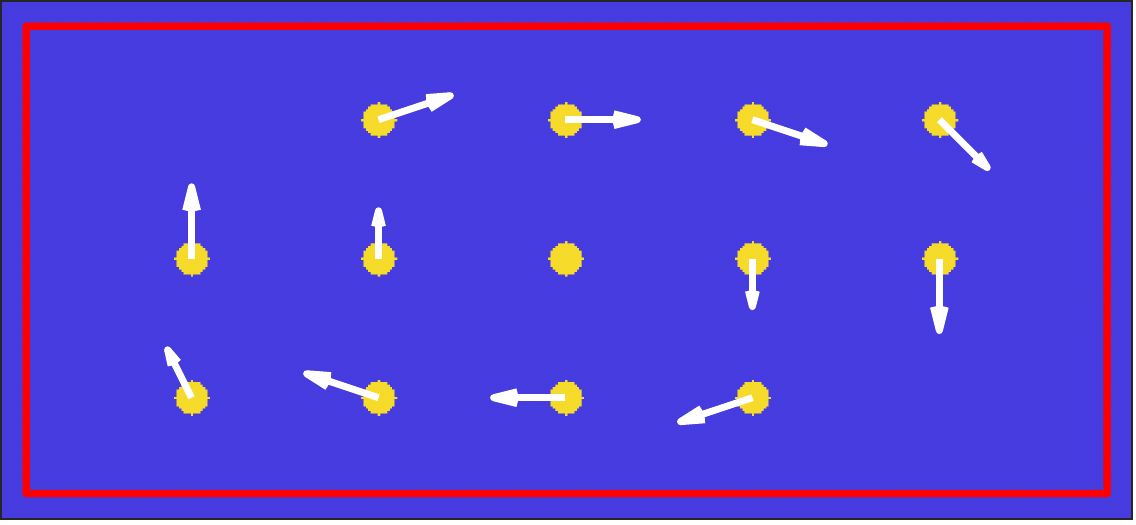}}
			
			\subfigure{
				\includegraphics[width=0.95\columnwidth]{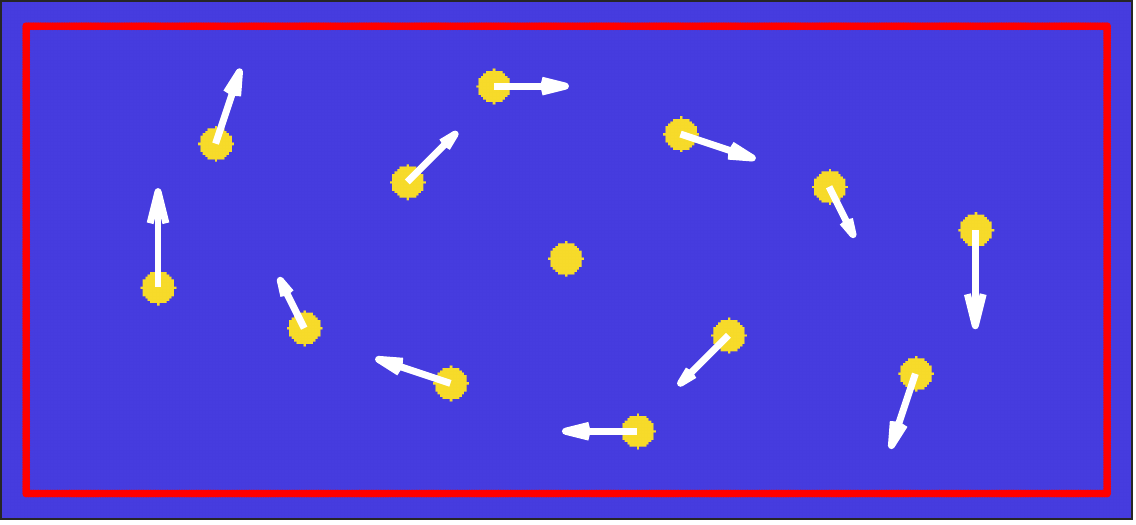}}
			
			\subfigure{
				\includegraphics[width=0.95\columnwidth]{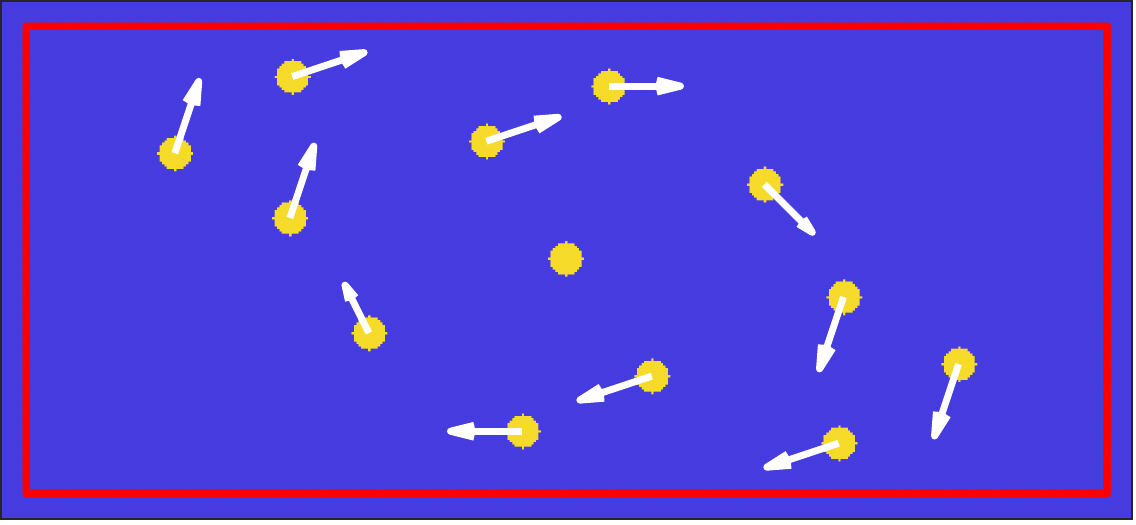}}
			
		\end{minipage}
		\begin{minipage}[t]{0.3\textwidth}
			\subfigure{
				\includegraphics[width=0.95\columnwidth]{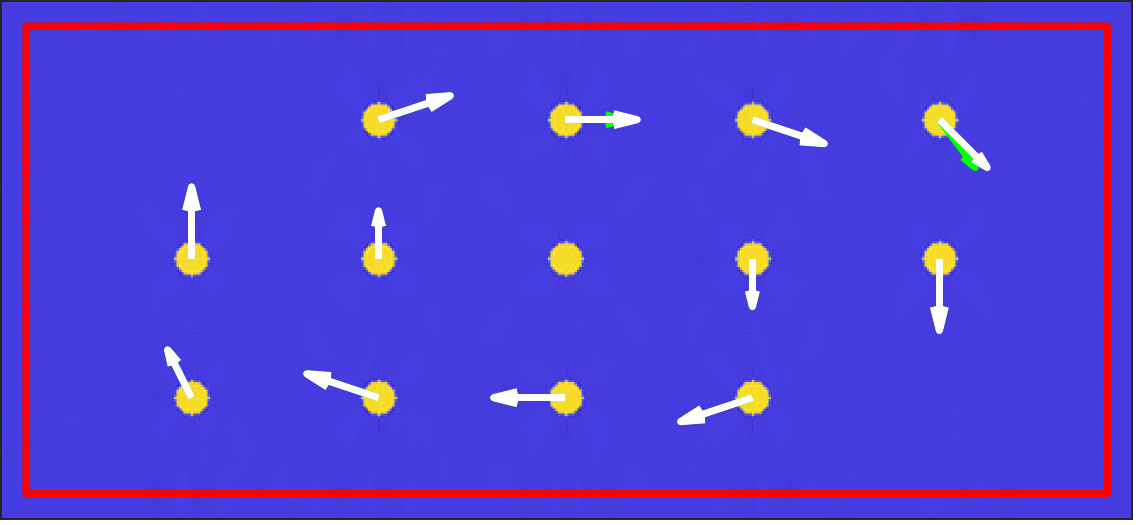}}
			
			\subfigure{
				\includegraphics[width=0.95\columnwidth]{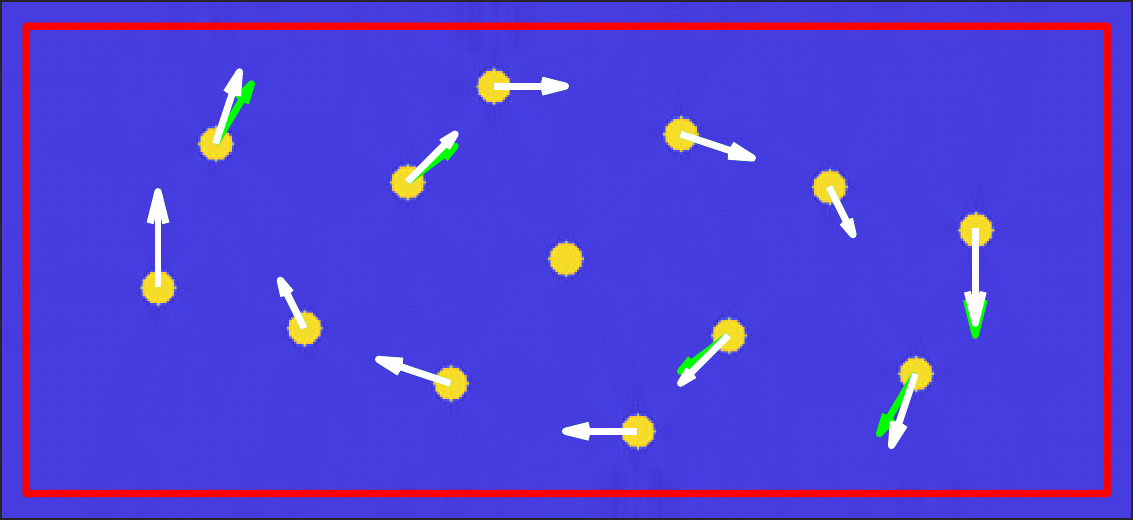}}
			
			\subfigure{
				\includegraphics[width=0.95\columnwidth]{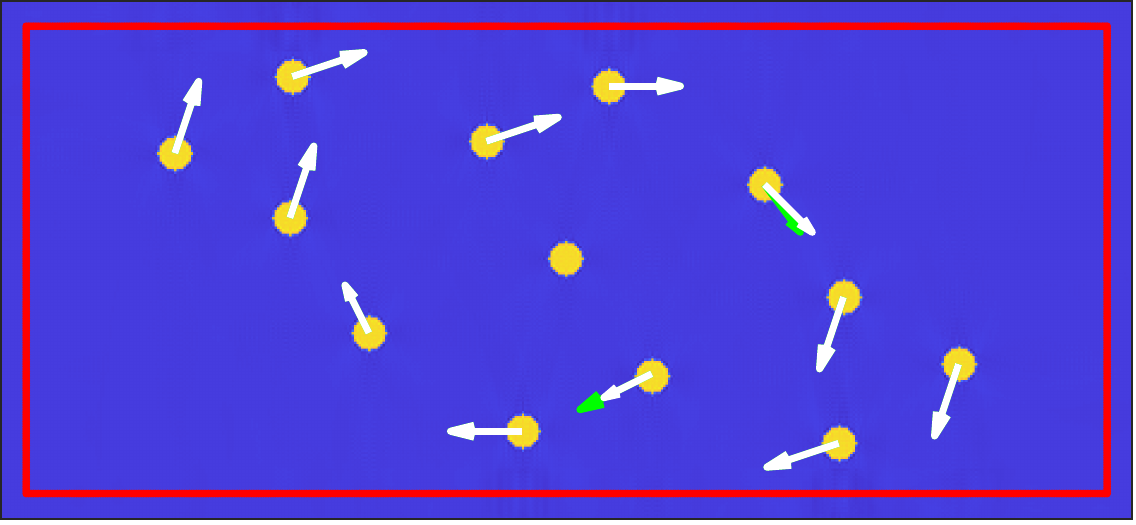}}
			
		\end{minipage}
		\begin{minipage}[t]{0.3\textwidth}
			\subfigure{
				\includegraphics[width=0.95\columnwidth]{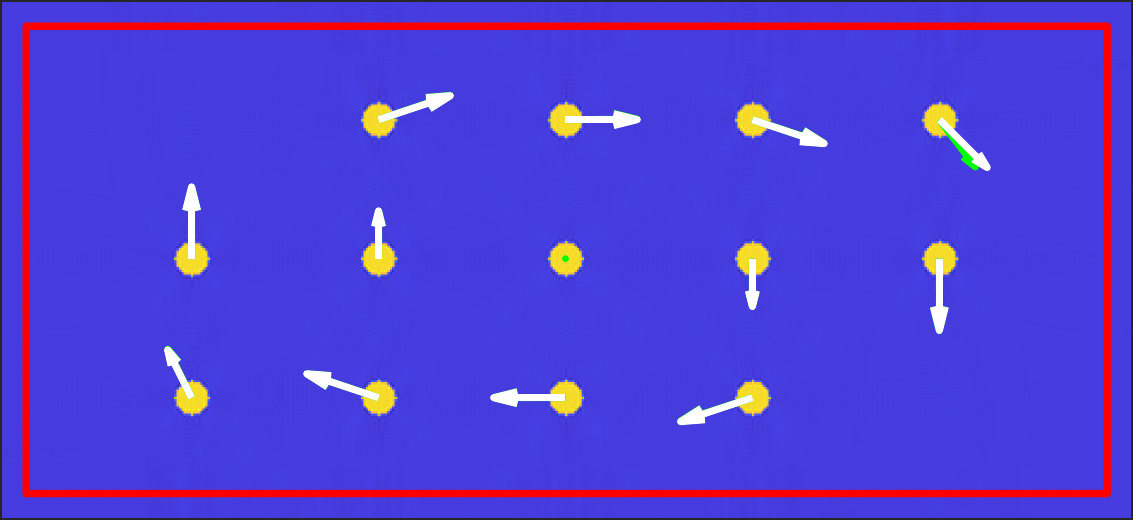}}
			
			\subfigure{
				\includegraphics[width=0.95\columnwidth]{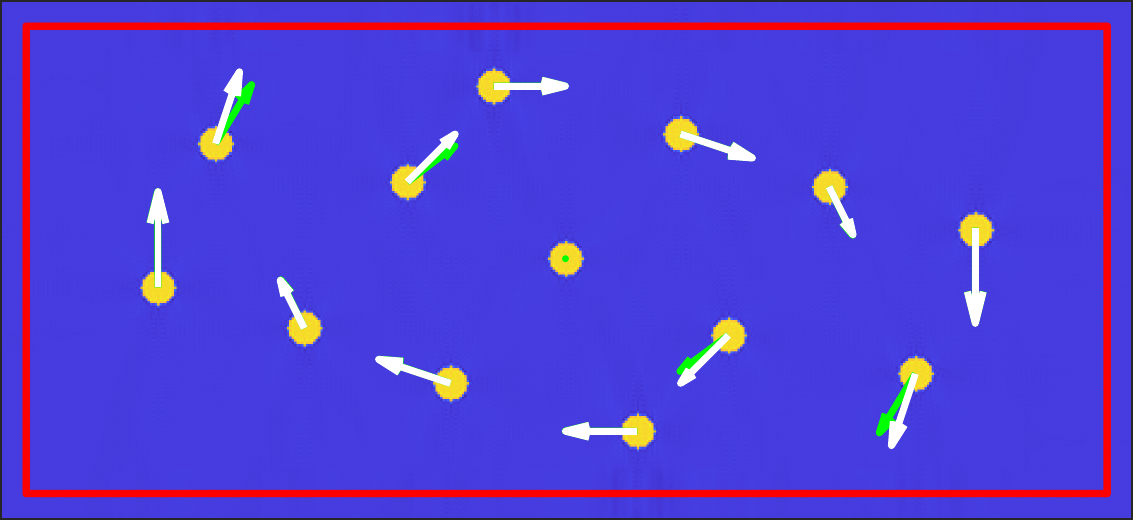}}
			
			\subfigure{
				\includegraphics[width=0.95\columnwidth]{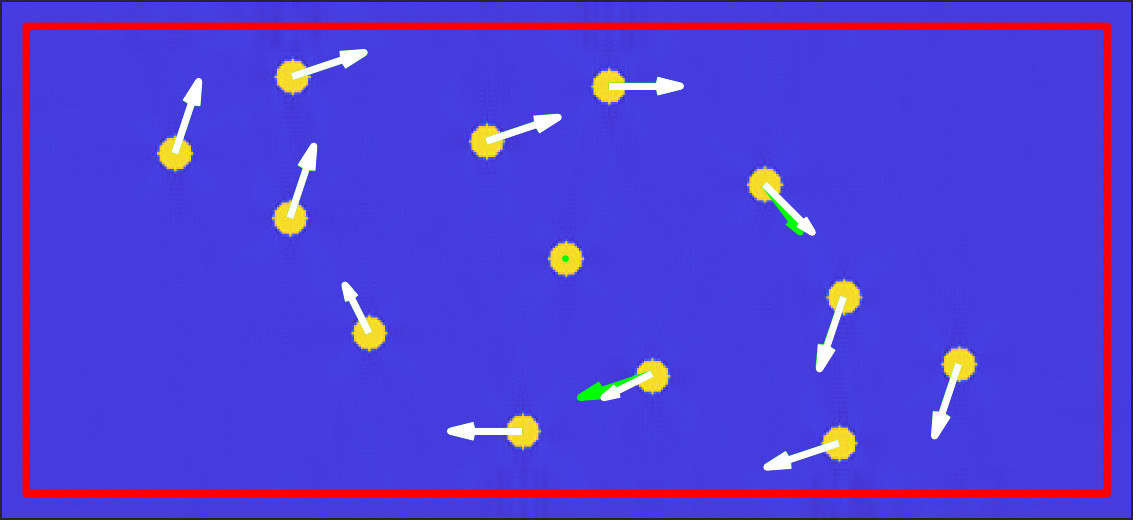}}
			
		\end{minipage}
		\begin{minipage}[t]{0.07\textwidth}
			\subfigure{\includegraphics[width=0.73\columnwidth]{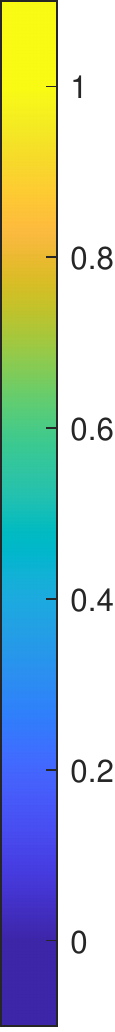}}
		\end{minipage}
		
		\caption{The contours of $f(x)$, the particle velocity, and the inversion for the local velocity of flow. The figures in the left column are exact figures of moving particles at several time points. The figures in the middle and right columns are both inverse results of moving particles from observations. The white and green arrows represent the exact velocity vectors and the approximated velocity vectors in the center of the disk particles at the current time. The approximated velocity vectors in the middle and right columns are computed using the nearest point searching and optical flow methods, respectively.}
		\label{invFew}
	\end{figure}
	
	\subsubsection{Inversion for the whole flow field}
	The following example shows the performance of our model for recovering the whole flow field. We select an unsteady flow field, the von K\'{a}rm\'{a}n vortex street, to evaluate the performance of our model. In the following example, more but smaller particles are placed in the flow field to simulate the practical measurements. In this setting, it is suitable to use the optical flow method to recover the whole flow field once we obtain the positions of particles on every frame since the smaller size makes it more difficult to identify the corresponding particles when applying the method of nearest point searching.
	
	The von K\'{a}rm\'{a}n vortex street was studied extensively by Theodore von K\'{a}rm\'{a}n in \cite{karmanAerodynamics1963}. Unlike the Taylor-Green vortex, the von K\'{a}rm\'{a}n vortex street does not have a closed formula. Therefore, the numerical simulation of dynamic flow is applied to generate a von K\'{a}rm\'{a}n vortex street. The simulation of dynamic flow is based on the lattice Boltzmann method (LBM), which is implemented by the Palabos library \cite{lattPalabosParallelLattice2021} with C\texttt{++}.
	We place 150 small particles of 0.01 m diameter in the simulated von K\'{a}rm\'{a}n vortex street; see the left figure in \Fref{invInKarman}. When obtaining the observations, the inverse source problem is solved to get the position of every particle at the time points of registering data; see the right figure in \Fref{invInKarman}. 
	
	Finally, the optical flow algorithm is applied to recover the vector field of the von K\'{a}rm\'{a}n vortex street from the positions of particles at each time section; see \Fref{invWholeKarman}. One can observe that the inversions of the particles at every time section are pretty accurate compared with the ground truth, thanks to the well-conditioned property of our model. They demonstrate the excellent effect of recovering the flow field.
	
			%
			%
		%
	
	\begin{figure}
		\centering
		\begin{minipage}{0.32\textwidth}
			%
			
			\subfigure{
				\includegraphics[width=0.95\columnwidth]{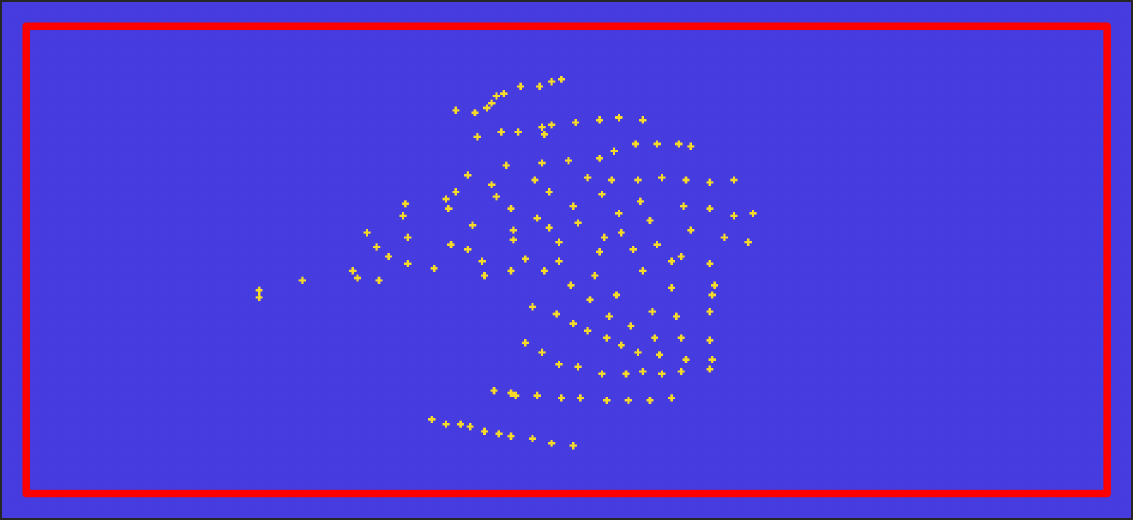}}
			
		\end{minipage}
		\begin{minipage}{0.32\textwidth}
			%
			
			\subfigure{
				\includegraphics[width=0.95\columnwidth]{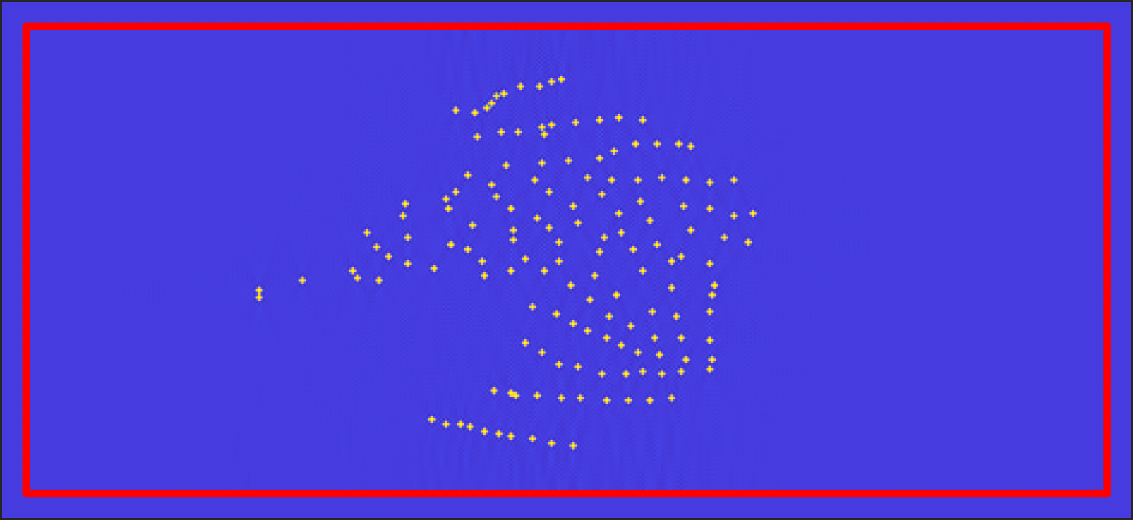}}
			
		\end{minipage}
		
		\caption{The contours of $f(x)$, i.e., particles inversion in the von K\'{a}rm\'{a}n vortex street at a certain time section. The figure on the left is the image of moving particles in von K\'{a}rm\'{a}n vortex street at a certain time section. The figures on the right is the reconstructed image of particles from data collected on receivers which are red points on images.}
		\label{invInKarman}
	\end{figure}
	
	\begin{figure}
		\centering

		\begin{minipage}{0.22\textwidth}
			\subfigure{
				\includegraphics[width=0.95\columnwidth]{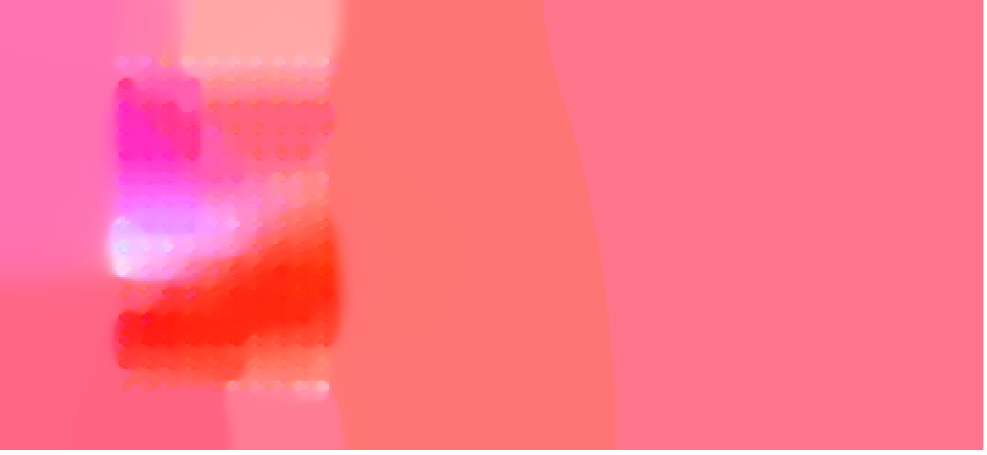}}
			\subfigure{
				\includegraphics[width=0.95\columnwidth]{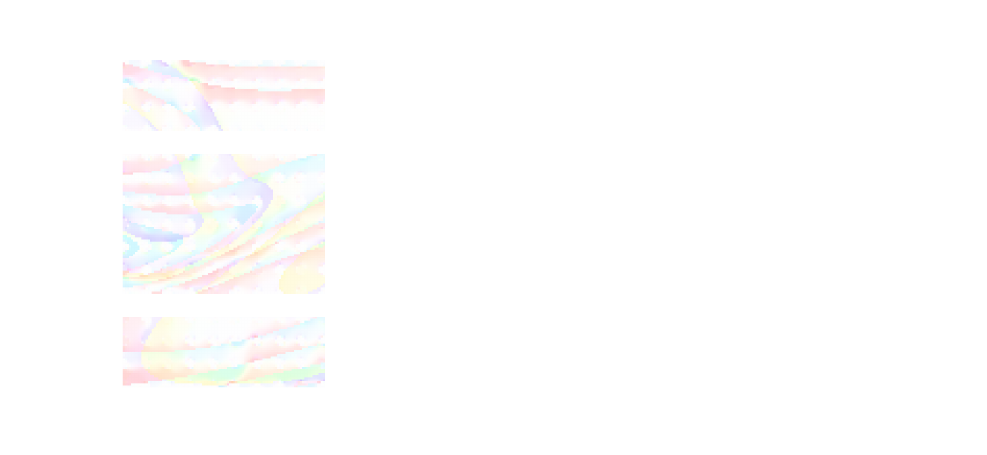}}
		\end{minipage}
		\begin{minipage}{0.22\textwidth}
			\subfigure{
				\includegraphics[width=0.95\columnwidth]{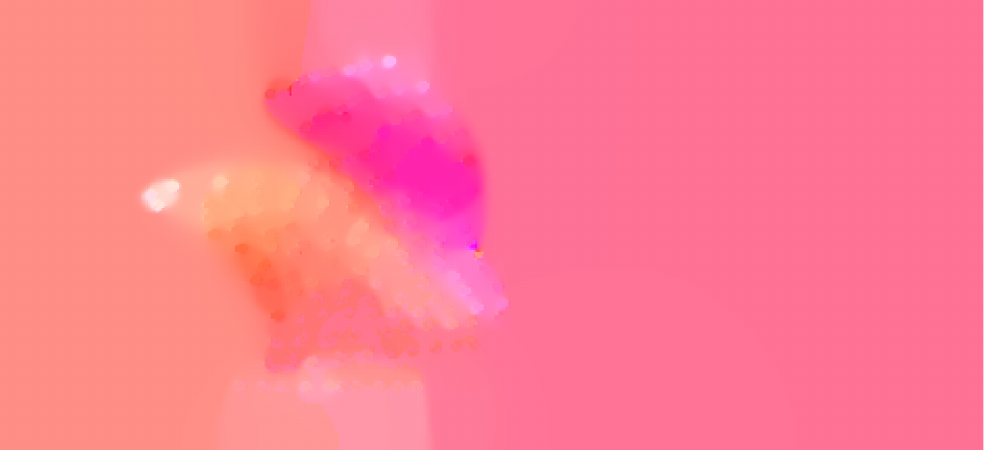}}
			
			\subfigure{
				\includegraphics[width=0.95\columnwidth]{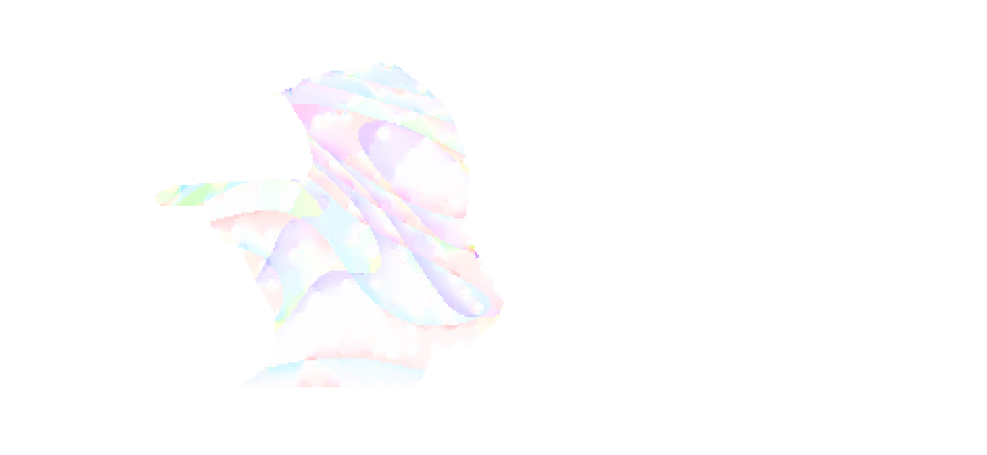}}
		\end{minipage}
		\begin{minipage}{0.22\textwidth}
			
			\subfigure{
				\includegraphics[width=0.95\columnwidth]{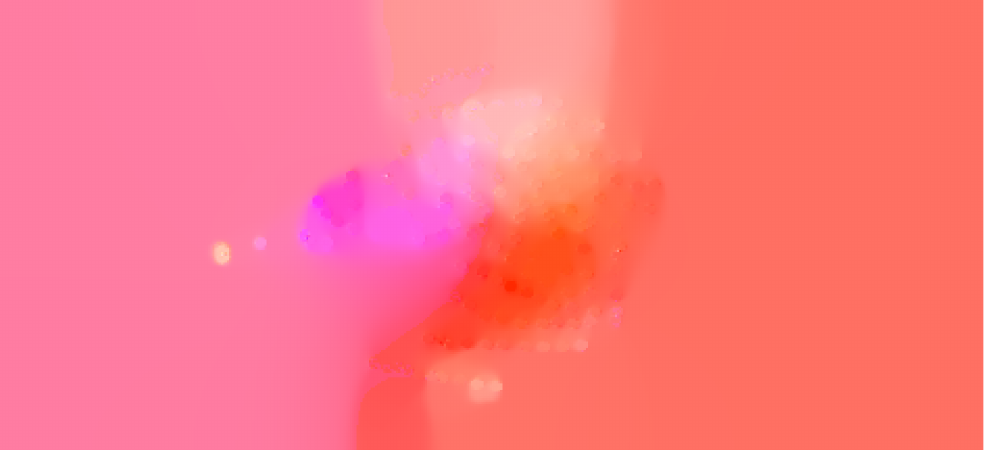}}
			
			\subfigure{
				\includegraphics[width=0.95\columnwidth]{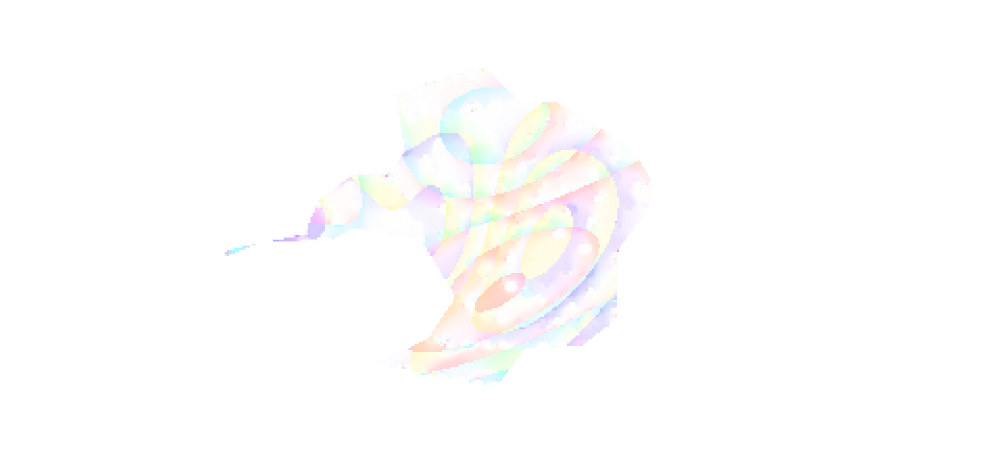}}
		\end{minipage}
		\begin{minipage}{0.22\textwidth}
			\subfigure{
				\includegraphics[width=0.95\columnwidth]{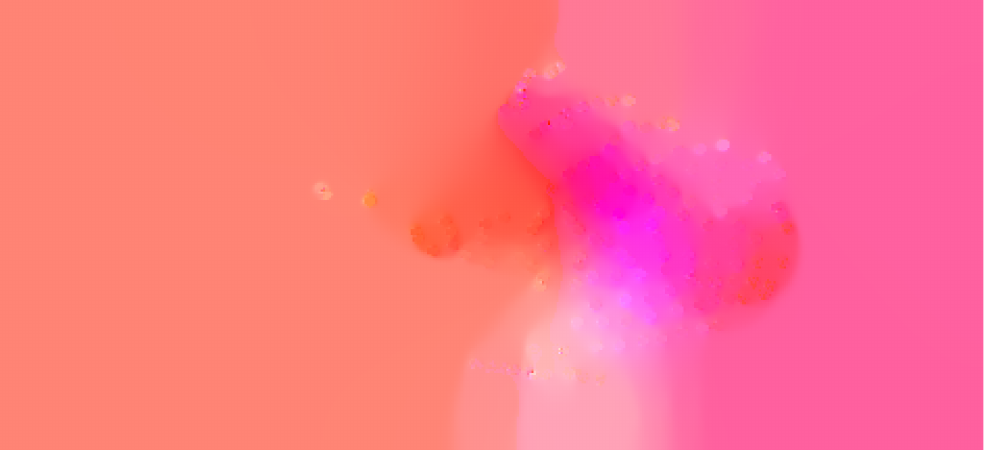}}
			
			\subfigure{
				\includegraphics[width=0.95\columnwidth]{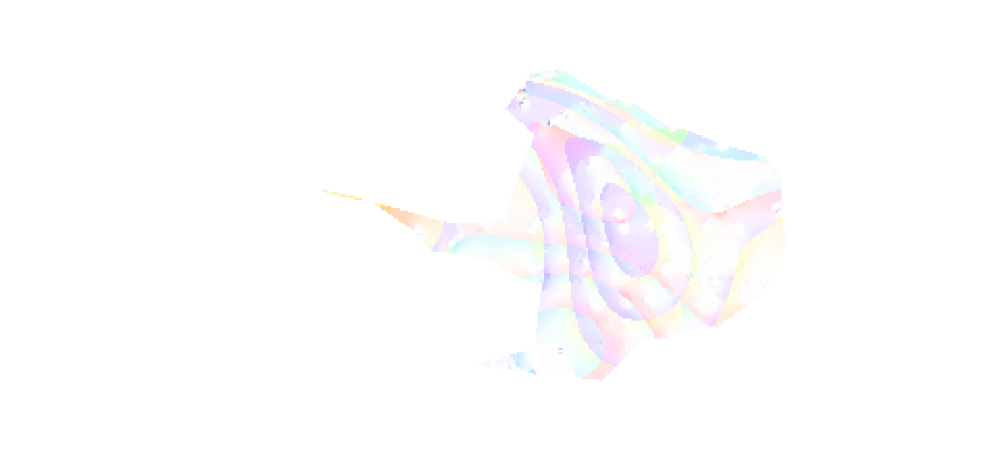}}
		\end{minipage}
		\begin{minipage}{0.05\textwidth}
			
			\vspace{60pt}
			\subfigure{
				\includegraphics[width=0.8\columnwidth]{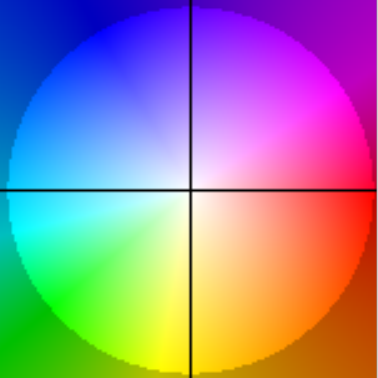}}
		\end{minipage}
		
		
		\caption{Inversion of the von K\'{a}rm\'{a}n vortex street. The figures in the top row show reconstructions of the flow field by applying the optical flow method. The figures in the bottom row are the subtractions between the true images and calculated images in the domain containing particles. The hue codes the direction, and the saturation codes the length.}
		\label{invWholeKarman}	
	\end{figure}
	\subsection{Comparison with the virtual ADCP}
	In this part, we compare the relative error of spatial averaging of the velocity between the virtual ADCP (V-ADCP) and our method using the von K\'{a}rm\'{a}n vortex street.
	The angle between the beams and the vertical is 20\textdegree~and each V-ADCP cell is defined as a volume delimited by a frustum of 4\textdegree ~cone.
	Similar to real ADCP, the spatial average velocity is calculated on each beam cell. Then the average of these velocities on all cells is regarded as the measurement of the V-ADCP in two dimensions. Assume that the V-ADCP measures the averaged velocity of the flow field precisely, which can be computed by the averaging formula using the numerical values of the flow. For our model, we place 450 particles of 0.01 m diameter in the von K\'{a}rm\'{a}n vortex street to cover the region of interest. Subsequently, the calculated results at a certain moment from the V-ADCP and our approach are compared with the exact spatial averaging of the flow velocity. The relative errors between the computed results and the exact averaged velocity are respectively 2.60\% and 10.05\%, which shows that the measurements from our novel model are more accurate than the conventional measurement techniques.
	
	\section{Conclusions\label{sec:conclusion}}
	In this work, we propose a novel inverse source problem formulation for the flow measurement. To the best of our knowledge, it is the first formulation based on inverse source problems. 
	We review the measurement theory of several instruments designed for flow measurement. The Doppler shift effect is one of the fundamental theories. 
	Employing the fact that the phenomenon of ultrasonic waves propagating in the water is governed by the wave equation, we take advantage of the inverse source problems to describe the measurement processes. In both theoretical and numerical terms, our novel formulation shows superior reconstruction accuracy. Moreover, only a few iterative steps are needed for a highly accurate recovery, thanks to the well-conditioned property of our model.
	
	We must emphasize that our novel model can be applied to various scenarios. 
	In addition to the examples shown in the paper, more experiments can be explored to investigate the effect of different parameter choices in our model, such as various types of the emitted source wave and non-constant sound speed. 
	Furthermore, our model has the capacity to design and simulate the performances of various instruments.
	
	
	
	Future work will also focus on simulating this model in three dimensions numerically and applying this model to practical terms, which means calculating the whole velocity field of flow from practical observations. Moreover, it is well-known that the flow slightly influences the speed of sound propagating in the water. Thus the coefficient $c$ in the wave equation would depend on position $x$ and time $t$. Although this is a minor change, the formulation will be more accurate when considering the inverse source problems for anisotropic wave equations. Developing an efficient algorithm and simulating the model based on anisotropic wave equations will be crucial follow-up works to the current method. Further theoretical analysis of the uniqueness and stability of the inverse source problem with anisotropic coefficients are also expected. 
	
	
	\bibliographystyle{plain}
	\bibliography{references}
	
\end{document}